\documentclass[12pt]{amsart}

\usepackage{amsmath,amsfonts,amssymb,mathabx,latexsym,mathtools}
\usepackage{enumitem}
\usepackage[usenames, dvipsnames]{xcolor}
\usepackage{hyperref}
\usepackage{ytableau}
\usepackage{centernot}
\usepackage{shuffle}
\usepackage{tikz-cd}
\usepackage{tikz}
\usetikzlibrary{calc, shapes, backgrounds,arrows,positioning,plotmarks}
\tikzset{>=stealth',
  head/.style = {fill = white, text=black},
  plaque/.style = {draw, rectangle, minimum size = 10mm, fill=white}, 
   posplaque/.style = {draw, star,star points=7,star point ratio=0.8, minimum size = 10mm, fill=white}, 
   Kplaque/.style = {draw, rectangle, minimum size = 10mm, fill=Gainsboro}, 
  newplaque/.style = {draw=red, ellipse, minimum size = 10mm,ultra thick, fill=white}, 
  posex/.style={->,thick},
  lift/.style={right hook->},
  beta/.style={dashed, <->},
  pil/.style={->,thick},
  junct/.style = {draw,circle,inner sep=0.5pt,outer sep=0pt, fill=black}
  }
  
\newcommand{\qcolor}[1]{\textcolor{violet}{#1}}
\newcommand{\scolor}[1]{\textcolor{orange}{#1}}
\newcommand{\pcolor}[1]{\textcolor{Green}{#1}}

\newtheorem{theorem}{Theorem}[section]

\newtheorem{proposition}[theorem]{Proposition}
\newtheorem{corollary}[theorem]{Corollary}
\newtheorem{conjecture}[theorem]{Conjecture}

\theoremstyle{definition}
\newtheorem{definition}[theorem]{Definition}
\newtheorem{examplex}[theorem]{Example}
\newenvironment{example}
  {\pushQED{\qed}\examplex}
  {\popQED\endexamplex}

\theoremstyle{remark}
\newtheorem{remark}[theorem]{Remark}

\numberwithin{equation}{section}

\newcommand{\SSYT}{\ensuremath{\mathrm{SSYT}}}

\newcommand{\id}{\ensuremath{\mathrm{id}}}

\newcommand{\xx}{\ensuremath{\mathbf{x}}}

\newcommand{\wt}{\ensuremath{\mathrm{wt}}}

\newcommand{\sort}[1]{\overleftarrow{#1}}
\newcommand{\reduced}{\ensuremath{\mathrm{Red}}}
\newcommand{\rev}{\ensuremath{\mathrm{rev}}}
\newcommand{\lswap}{\ensuremath{\mathrm{lswap}}}
\newcommand{\Qlswap}{\ensuremath{\mathrm{Qlswap}}}

\newcommand{\sym}{\ensuremath{\mathrm{Sym}}}
\newcommand{\qsym}{\ensuremath{\mathrm{QSym}}}
\newcommand{\asym}{\ensuremath{\mathrm{ASym}}}
\newcommand{\poly}{\ensuremath{\mathrm{Poly}}}

\newcommand{\schub}{\ensuremath{\mathfrak{S}}}
\newcommand{\groth}{\ensuremath{\overline{\mathfrak{S}}}}
\newcommand{\slide}{\ensuremath{\mathfrak{F}}}
\newcommand{\glide}{\ensuremath{\overline{\mathfrak{F}}}}
\newcommand{\mslide}{\ensuremath{\mathfrak{M}}}

\newcommand{\atom}{\ensuremath{\mathfrak{A}}}
\newcommand{\lascouxatom}{\ensuremath{\overline{\mathfrak{A}}}}
\newcommand{\kaon}{\ensuremath{\overline{\mathfrak{P}}}}
\newcommand{\particle}{\ensuremath{\mathfrak{P}}}
\newcommand{\monomial}{\ensuremath{\mathfrak{X}}}
\newcommand{\fundamental}{\ensuremath{F}}

\newcommand{\key}{\ensuremath{\mathfrak{D}}}
\newcommand{\qkey}{\ensuremath{\mathfrak{Q}}}
\newcommand{\lascoux}{\ensuremath{\overline{\mathfrak{D}}}}
\newcommand{\qlascoux}{\ensuremath{\overline{\mathfrak{Q}}}}
\newcommand{\schur}{\ensuremath{s}}
\newcommand{\sgroth}{\ensuremath{\overline{s}}}
\newcommand{\qschur}{\ensuremath{S}}

\newcommand{\SSF}{\ensuremath{\mathsf{SST}}}
\newcommand{\ASSF}{\atom\SSF}

\newcommand{\LSSF}{\particle\SSF}
\newcommand{\HSSF}{\ensuremath{\mathsf{H}}\SSF}

\newcommand{\KSSF}{\key\SSF}

\newcommand{\PD}{\ensuremath{\mathrm{PD}}}
\newcommand{\QPD}{\ensuremath{\mathrm{QPD}}}
\newcommand{\KD}{\ensuremath{\mathrm{KD}}}
\newcommand{\RD}{\ensuremath{\mathrm{RD}}}

\newcommand{\Sh}{\ensuremath{\mathrm{Sh}}}
\newcommand{\Des}{\ensuremath{\mathrm{Des}}}
\newcommand{\Runs}{\ensuremath{\mathrm{Runs}}}
\newcommand{\Bumpruns}{\ensuremath{\mathrm{BumpRuns}}}
\newcommand{\Bump}{\ensuremath{\mathrm{Bump}}}
\newcommand{\Comp}{\ensuremath{\mathrm{Comp}}}

\newcommand{\excise}[1]{}

\newcommand\turn{
\begin{picture}(10,10)
\thicklines
\put(0,10){\oval(10,10)[br]}
\put(10,0){\oval(10,10)[tl]}
\end{picture}}

\newcommand\tail{
\begin{picture}(10,10)
\thicklines
\put(0,10){\oval(10,10)[br]}
\end{picture}}

\newcommand\cross{
\begin{picture}(10,10)
\thicklines
\put(5,0){\line(0,1){10}}
\put(0,5){\line(1,0){10}}
\end{picture}}

\newcommand\gridify[1]{\vbox to 10\unitlength{\vss\hbox to 10\unitlength{\hss$_{#1}$\hss}\vss}}

\newcommand\pipes[1]{\vtop{\let\\=\cr
\setlength\baselineskip{-10000pt}
\setlength\lineskiplimit{10000pt}
\setlength\lineskip{0pt}
\halign{&\gridify{##}\cr#1\crcr}}}

\newlength\cellsize 

\savebox2{%
\begin{picture}(16,16)
\put(0,0){\line(1,0){16}}
\put(0,0){\line(0,1){16}}
\put(16,0){\line(0,1){16}}
\put(0,16){\line(1,0){16}}
\end{picture}}

\newcommand\cellify[1]{\def\thearg{#1}\def\nothing{}%
\ifx\thearg\nothing\vrule width0pt height\cellsize depth0pt%
  \else\hbox to 0pt{\usebox2\hss}\fi%
  \vbox to 16\unitlength{\vss\hbox to 16\unitlength{\hss$#1$\hss}\vss}}

\newcommand\tableau[1]{\vtop{\let\\=\cr
\setlength\baselineskip{-16000pt}
\setlength\lineskiplimit{16000pt}
\setlength\lineskip{0pt}
\halign{&\cellify{##}\cr#1\crcr}}}

\savebox4{
\begin{picture}(16,16)
\put(8,8){\circle{16}}
\end{picture}}

\newcommand{\cir}[1]{\def\thearg{#1}\def\nothing{}%
\ifx\thearg\nothing\vrule width0pt height\cellsize depth0pt%
  \else\hbox to 0pt{\usebox4\hss}\fi%
  \vbox to 16\unitlength{\vss\hbox to 16\unitlength{\hss$#1$\hss}\vss}}

\newcommand\nocellify[1]{\def\thearg{#1}\def\nothing{}%
\ifx\thearg\nothing\vrule width0pt height\cellsize depth0pt%
  \else\hbox to 0pt{\hss}\fi%
  \vbox to 16\unitlength{\vss\hbox to 16\unitlength{\hss$#1$\hss}\vss}}

\newcommand\notableau[1]{\vtop{\let\\=\cr
\setlength\baselineskip{-16000pt}
\setlength\lineskiplimit{16000pt}
\setlength\lineskip{0pt}
\halign{&\nocellify{##}\cr#1\crcr}}}

\definecolor{boxgray}{gray}{.7}

\begin{document}


\title{Asymmetric function theory}  

\author[O. Pechenik]{Oliver Pechenik}
\address[OP]{Department of Mathematics, University of Michigan, Ann Arbor, MI 48109, USA}
\email{pechenik@umich.edu}

\author[D. Searles]{Dominic Searles}
\address[DS]{Department of Mathematics and Statistics, University of Otago, Dunedin 9016, New Zealand}
\email{dominic.searles@otago.ac.nz}


\date{\today}



\begin{abstract}
The classical theory of symmetric functions has a central position in algebraic combinatorics, bridging aspects of representation theory, combinatorics, and enumerative geometry. More recently, this theory has been fruitfully extended to the larger ring of quasisymmetric functions, with corresponding applications. Here, we survey recent work extending this theory further to general asymmetric polynomials.
\end{abstract}

\maketitle

%
\section{The three worlds: symmetric, quasisymmetric, and general polynomials}
%
\label{sec:introduction}
One of the gems of 20th-century mathematics is theory of \emph{symmetric functions} and \emph{symmetric polynomials}, as expounded in the classic textbooks \cite{EC2, Manivel, Macdonald}. First, we will review aspects of this theory. Then, we discuss the more general theory of \emph{quasisymmetric functions and polynomials}, a very active area of contemporary research. Finally, we turn to the combinatorial theory of general asymmetric polynomials. While this seems naively like a very simple object, the polynomial ring turns out to have a rich and beautiful combinatorial structure analogous to that of the symmetric and quasisymmetric worlds, but far less explored.

In each world, we will consider a variety of additive bases. The power of the combinatorial theory comes from having the following three characteristics: 
\begin{enumerate}
\item positive combinatorial rules for change of basis,
\item positive combinatorial multiplication rules in various bases, and
\item algebraic/geometric interpretations of the basis elements.
\end{enumerate}

\section{The symmetric world}\label{sec:sym}
 Consider the $\mathbb{Z}$-algebra $\poly_n \coloneqq \mathbb{Z}[x_1, \dots, x_n]$ of integral multivariate polynomials. It carries a natural action of the symmetric group $\mathcal{S}_n$ on $n$ letters, where the simple transposition $(i \; i+1)$ acts on $f \in \poly_n$ by swapping the variables $x_i$ and $x_{i+1}$. Let $\sym_n \coloneqq \poly_n^{\mathcal{S}_n}$, the $\mathcal{S}_n$-invariants. It is easy to see that $\sym_n$ is a subring of $\poly_n$; we call it the {\bf ring of symmetric polynomials in $n$ variables}. $\sym_n$ is moreover a graded ring, inheriting the grading by degree from $\poly_n$. We denote the degree $m$ homogeneous piece of a graded ring $R$ by $R^{(m)}$.

For $m \leq n$, we can map $\sym_n$ onto $\sym_m$ by setting the last $n-m$ variables equal to $0$. The inverse limit of the $\{ \sym_n \}$ with respect to these restriction maps is called the {\bf ring of symmetric functions} $\sym$, although its elements are not functions, but rather formal power series in infinitely-many variables. Classically, one generally prefers to study $\sym$; however, we will usually prefer the essentially equivalent theory of $\sym_n$, as it extends more naturally to the asymmetric setting that is our focus.

We will consider four of the most important additive bases of $\sym_n$: the monomial, elementary, homogeneous, and Schur bases.

\begin{figure}[h]
\begin{tikzpicture}
\node[posplaque] (schur) {\scolor{$\schur_\lambda$}};
\node[below left=0.5 and 1 of schur, posplaque] (e) {\scolor{$e_\lambda$}};
\node[above left =0.5 and 1 of schur, posplaque] (h) {\scolor{$h_\lambda$}};
\node[right=1 of schur,posplaque] (m) {\scolor{$m_\lambda$}};
  \begin{scope}[nodes = {draw = none}]
    \path (e) edge[posex]   (schur)
    (h) edge[posex]  (schur)
    (schur) edge[posex] (m)
      ;
  \end{scope}
\end{tikzpicture}
\caption{The four bases of $\sym_n$ considered here. The arrows denote that the basis at the head refines the basis at the tail. All four bases have positive structure coefficients.}\label{fig:sym}
\end{figure}
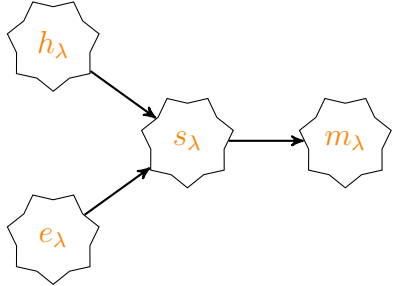

For a {\bf weak composition} $a$ (i.e., an infinite sequence of nonnegative integers with finite sum), define a monomial 
\[
\xx^a \coloneqq x_1^{a_1}x_2^{a_2} \cdots.
\] 
For a {\bf partition} $\lambda$ (i.e., a weakly decreasing weak composition), let 
\[ m_\lambda \coloneqq \sum_a \xx^a,
\]
where the sum is over all distinct weak compositions that can be obtained by rearranging the parts of $\lambda$. If $\xx^a$ is a monomial of the symmetric function $f$, then necessarily $\xx^b$ is also a monomial of $f$ for every $b$ that can be obtained by rearranging the parts of $a$. Thus $f$ can be written uniquely as a finite sum of the {\bf monomial symmetric functions} $m_\lambda$. Therefore $\{ m_\lambda \}$ is a $\mathbb{Z}$-linear basis of $\sym$ and the dimension of $\sym^{(m)}$ is the number of partitions of $m$.

We may consider a second action of $\mathcal{S}_n$ on $\poly_n$ where a permutation acts by permuting variables and then multiplying by the sign of the permutation. Note that this merely amounts to twisting the orginal action by tensoring with the $1$-dimensional sign representation of $\mathcal{S}_n$. The invariants of this twisted action are the {\bf alternating polynomials in $n$ variables}, $v_n \sym_n$. These are precisely the polynomials where setting any two variables equal yields $0$. The sum of two alternating polynomials is alternating, but the product is generally not. Hence $v_n \sym_n$ is not a subring of $\poly_n$, although it is a module over $\sym_n$. As with $\sym_n$, $v_n \sym_n$ is graded by degree, although for technical reasons one might prefer to shift the degree by $\binom{n}{2}$.  

Let $v_n := \prod_{1 \leq i < j \leq n} (x_i - x_j)$ be the {\bf Vandermonde determinant}. This is an alternating polynomial and moreover divides every other alternating polynomial. The quotients are necessarily symmetric. Hence every alternating polynomial can be written as $v_n$ times a symmetric polynomial. (This fact justifies the notation $v_n \sym_n$ for the module of alternating polynomials.) 

For a weak composition $a$ of length $n$, define \[ \tilde{j}_a := \sum_{\sigma \in S_n} {\rm sgn}(\sigma) x^{\sigma(a)}. \] Note that if $\xx^a$ is a term of the alternating polynomial $f$, then so is every other term of $\tilde{j}_a$. Moreover if $a$ has any repeated parts, then clearly $\tilde{j}_a =0$. Hence $v_n \sym_n$ has a natural basis of polynomials
$ \tilde{j}_{\theta},$ for $\theta$ ranging over {\bf strict partitions}, that is partitions with distinct parts.

Every strict partition may be written uniquely as $\delta + \lambda$, where $\delta = (n-1, n-2, \dots, 0)$, $\lambda$ is a partition, and the sum is componentwise. We write $j_\lambda \coloneqq \tilde{j}_{\delta + \lambda}$, to obtain a basis of $v_n \sym$ indexed by partitions. That is, the dimension of the space of alternating polynomials of degree $m + \binom{n}{2}$ equals the dimension of the space of symmetric polynomials of degree $m$. Indeed, we can even identify the isomorphism; it is just multiplication by $v_n = \tilde{j}_\delta = j_{(0)}$. If we shifted the grading of $v_n\sym_n$ as suggested above (so that $v_n$ is in degree $0$), then multiplication by $v_n$ is an isomorphism $\sym_n \to v_n\sym_n$ of \emph{graded} $\sym$-modules.

The basis of $\sym_n$ obtained by pulling back the $j_\lambda$ basis of $v_n\sym_n$ is not the basis of monomial symmetric polynomials, but rather something more interesting. These important objects \[s_\lambda \coloneqq \frac{j_\lambda}{v_n}\] are called the {\bf Schur polynomials}. 

Although the Schur polynomials are clearly symmetric and hence can be expanded in the monomial basis, it is a remarkable surprise that these expansion coefficients are uniformly \emph{positive}. A recurring theme in this survey will be such instances of positive basis changes between {\it a priori} unrelated bases. 

A combinatorial formula manifesting the monomial-positivity of Schur polynomials was given by Littlewood. Given a partition $\lambda = (\lambda_1, \lambda_2, \ldots)$, we identify $\lambda$ with its English-orientation Young diagram, consisting of $\lambda_1$ left-justified boxes in the top row, $\lambda_2$ left-justified boxes in the second row, etc. A {\bf semistandard (Young) tableau} of shape $\lambda$ is an assignment of a positive integer to each box of the Young diagram such that the labels weakly increase left to right across rows and strictly increase down columns. The {\bf weight} of a tableau $T$ is the weak composition $\wt(T) \coloneqq (a_1, a_2, \ldots)$, where $a_i$ records the number of boxes labeled $i$.

\begin{theorem}[Littlewood]\label{thm:littlewood_formula}
$s_\lambda = \sum_{T \in \SSYT(\lambda)} \xx^{\wt(T)}$. \qed
\end{theorem}

\begin{example}
We have $s_{(2,1)}(x_1, x_2) = x_1^2 x_2 + x_1 x_2^2$, owing to the two semistandard tableaux
\[
\ytableaushort{11,2} \quad \ytableaushort{12,2}.
\]
\end{example}

We now turn to the last two bases of $\sym_n$ that we will consider.
For a partition $\lambda = (\lambda_1, \lambda_2, \dots)$, we define the {\bf elementary symmetric function} $e_\lambda$ by 
\[
e_\lambda \coloneqq \prod_i s_{\lambda_i}
\]
and the {\bf (complete) homogeneous symmetric function} by
\[
h_\lambda \coloneqq \prod_i s_{1^{\lambda_i}}.
\]

It is not obvious that either of these families yields a basis of $\sym_n$; nevertheless, each of them does, as was originally established by Isaac Newton. It also is not obvious that, like the Schur basis, the elementary and homogeneous bases expand positively in the $m_\lambda$. However, in fact, something much stronger is true: each $e_\lambda$ and each $h_\lambda$ is a positive sum of Schur polynomials. This positivity is a consequence of an even more remarkable positivity property:

\begin{theorem}\label{thm:Clmn_positive}
The Schur basis of $\sym_n$ has positive structure coefficients.
In other words, for any partitions $\lambda$ and $\mu$, the product $s_\lambda \cdot s_\mu$ expands as a positive sum of Schur polynomials. 
\end{theorem}

\begin{corollary}
For any $\lambda$, $e_\lambda$ and $h_\lambda$ are both Schur-positive, and hence monomial-positive.
\end{corollary}
\begin{proof}
By Theorem~\ref{thm:Clmn_positive}, any product of Schur polynomials is Schur-positive. Since $e_\lambda$ and $h_\lambda$ are defined as products of special Schur polynomials, they are then Schur-positive. By Theorem~\ref{thm:littlewood_formula}, Schur polynomials are monomial-positive. Hence, any Schur-positive polynomial, in particular $e_\lambda$ or $h_\lambda$, is also monomial-positive.
\end{proof}

By commutativity and the definitions, it is transparent that both the elementary and homogeneous bases of $\sym_n$ also have positive 
structure coefficients.

There are a variety of distinct ways to establish Theorem~\ref{thm:Clmn_positive}. The most fundamental explanations involve interpreting the theorem algebraically or geometrically. For example, one can establish that $\sym$ is isomorphic to the ring of polynomial representations of the general linear group in such a way that the Schur functions are in one-to-one correspondence with the irreducible representations. Under this identification, decomposing the tensor product of two irreducible representations into irreducibles corresponds to expanding the product of two Schur functions in the Schur basis. Hence, Theorem~\ref{thm:Clmn_positive} follows. 

Similarly, one can identify the Schur functions of homogeneous degree $k$ with the irreducible representations of the symmetric group $\mathcal{S}_n$ in such a way that multiplying Schur functions corresponds to taking an `induction product' of the corresponding representations. Again, since the induction product representation is necessarily a direct sum of irreducible representations, we recover Theorem~\ref{thm:Clmn_positive}. For more details on these representation-theoretic proofs, see, e.g., \cite{Fulton:book,Manivel}.

A geometric approach is to identify $\sym$ with the Chow ring of complex Grassmannians, the classifying spaces for complex vector bundles. A Grassmannian comes with a natural cell decomposition by certain subvarieties called \emph{Schubert varieties}, yielding an effective basis of the Chow ring. Under the identification with $\sym$, this basis corresponds to the Schur polynomials. Multiplying Schur polynomials then corresponds to intersection product on Schubert varieties, and again Theorem~\ref{thm:Clmn_positive} follows. For more details on these geometric constructions, see, e.g., \cite{Fulton:book,Manivel,Gillespie}.

The above interpretations of Theorem~\ref{thm:Clmn_positive} deepen its significance and provide relatively easy proofs. Nonetheless, Theorem~\ref{thm:Clmn_positive} is on its face a purely combinatorial statement and so one might hope it also had a purely combinatorial proof. Indeed, such a proof exists. Even better, it gives an explicit transparently-positive formula for the positive integers appearing in the Schur expansion. This formula can then be combined with the algebraic and geometric interpretations above to compute with and to better understand aspects of representation theory and enumerative geometry.

We write $\lambda \subseteq \nu$ to mean that the Young diagram of the partition $\lambda$ is a subset of that for $\nu$. The set-theoretic difference is called the {\bf skew Young diagram} $\nu / \lambda$. A {\bf skew semistandard tableau} is a filling of a skew diagram by positive integers, such that rows weakly increase and columns strictly increase. Define the content of a skew tableau as for tableaux of partition shape. The {\bf reading word} of a (skew) tableau $T$ is the word given by reading the rows of $T$ from top to bottom and from right to left (like the ordinary reading order in Arabic or Hebrew). We say that $T$ is {\bf Yamanouchi} if every initial segment of its reading word contains at least as many $i$s as $(i+1)$s, for each positive integer $i$.

\begin{theorem}[Littlewood-Richardson rule]\label{thm:LRrule}
For partitions $\lambda$ and $\mu$, we have 
\[
s_\lambda \cdot s_\mu = \sum_\nu c_{\lambda, \mu}^\nu s_\nu,
\]
where $c_{\lambda, \mu}^\nu$ counts the number of Yamanouchi semistandard tableaux of skew shape $\nu / \lambda$ and content $\mu$.
\end{theorem}

\begin{example}
To compute the structure coefficient $c_{(2,1),(2,1)}^{(3,2,1)}$ via Theorem~\ref{thm:LRrule}, we consider fillings of the skew shape $(3,2,1) / (2,1)$:
\[
\ydiagram{2+1,1+1,1}
\]
using two $1$s and one $2$. There are three such fillings
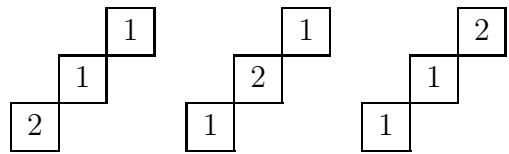
\[
\ytableaushort{\none \none 1, \none 1, 2} \quad \ytableaushort{\none \none 1, \none 2, 1} \quad \ytableaushort{\none \none 2, \none 1, 1}
\]
all of which are semistandard. However, only the first two are Yamanouchi, as the reading word of the third is $211$, which has a $2$ before any $1$. Hence, $c_{(2,1),(2,1)}^{(3,2,1)} = 2$.
\end{example}

\section{The quasisymmetric world}\label{sec:qsym}
In this section, we consider a third action of $\mathcal{S}_n$ on $\poly_n$. Here, the simple transposition $(i \; i+1)$ acts on $f \in \poly_n$ by swapping the variables $x_i$ and $x_{i+1}$ \emph{in only those terms not involving both variables}. The invariants of this action are the subalgebra $\qsym_n$ of {\bf quasisymmetric polynomials}.  Analogously to the symmetric case, one can also define the ring $\qsym$ of {\bf quasisymmetric functions} in infinitely-many variables as the inverse limit of the $\qsym_n$, but again our focus is on the essentially equivalent finite-variable case. For a much more detailed survey than we provide here of the state of the art in quasisymmetric function theory, see \cite{Mason}.

We will consider three important bases of $\qsym_n$: the monomial, fundamental and quasiSchur bases.

\begin{figure}[h]
\begin{tikzpicture}
\node[plaque] (quasischur) {\qcolor{$\qschur_\alpha$}};
\node[above=1 of quasischur,posplaque] (schur) {\scolor{$\schur_\lambda$}};
\node[right=1 of quasischur,posplaque] (fund) {\qcolor{$F_\alpha$}};
\node[right=1 of fund,posplaque] (mono) {\qcolor{$M_\alpha$}};
\node[above=1 of mono,posplaque] (monosym) {\scolor{$m_\lambda$}};
  \begin{scope}[nodes = {draw = none}]
    \path (quasischur) edge[posex]   (fund)
    (fund) edge[posex]  (mono)
     (schur) edge[posex] (quasischur)
     (schur) edge[posex] (monosym)
     (monosym) edge[posex] (mono)
      ;
  \end{scope}
\end{tikzpicture}
\caption{The three bases of $\qsym_n$ considered here, together with some bases of $\sym_n$ from Figure~\ref{fig:sym}. The arrows denote that the basis at the head refines the basis at the tail. The star-shaped nodes have positive structure coefficients.}\label{fig:qsym}
\end{figure}
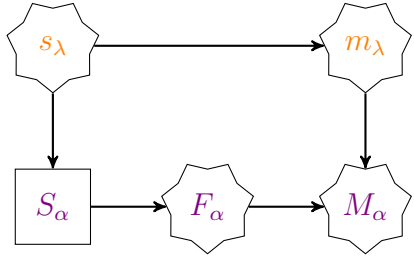

A {\bf strong composition} $\alpha$ is a finite sequence of positive integers; we identify $\alpha$ with the weak composition obtained by appending infinitely many $0$s to the end of $\alpha$. For any weak composition $a$, its {\bf positive part} is the strong composition $a^+$ given by deleting all $0$s.

For any strong composition $\alpha$, define the {\bf monomial quasisymmetric polynomial} $M_\alpha$ by
\[M_\alpha(x_1,\ldots , x_n) \coloneqq \sum_{b}{\mathbf x}^b \in \qsym_n, \]
where the sum is over all weak compositions $b$ with $b^+ = \alpha$ and whose entries after position $n$ are all zero. Clearly, the monomial quasisymmetric polynomials yield a basis of $\qsym_n$.

\begin{example}
We have \[M_{13}(x_1,x_2,x_3) = {\mathbf x}^{130} + {\mathbf x}^{103} + {\mathbf x}^{013} \in \qsym_3.\]
Note in particular that this polynomial is not an element of $\sym_3$.
\end{example}

It is clear that the monomial basis of $\qsym_n$ must have positive structure coefficients, like all the bases of $\sym_n$ discussed in Section~\ref{sec:sym}. However, these structure coefficients are slightly more interesting that those for the monomial symmetric functions. They are given by the overlapping shuffles of M.~Hazewinkel \cite{Hazewinkel}, which we now recall.

Let $A$ and $B$ be words in disjoint alphabets with $A$ of length $m$ and $B$ of length $n$.  An {\bf overlapping shuffle} of $A$ and $B$ is a surjection
\[
t : \{1, 2,  \dots, m+n \} \to \{1, 2, \dots, k\}
\]
(for some $\max \{m,n\} \leq k \leq m+n$) such that \[t(i) < t(j) \text{ whenever } i < j \leq m \text{ or } m < i < j.\]
We write $A \shuffle_o B$ for the {\bf overlapping shuffle product} of $A$ and $B$, the formal sum of all overlapping shuffles. The overlapping shuffle product $\alpha \shuffle_o \beta$ of two strong compositions $\alpha$ and $\beta$ is given by treating the strong compositions as words in disjoint alphabets. Here, we identify an overlapping shuffle $t : \{1, 2,  \dots, m+n \} \to \{1, 2, \dots, k\}$ of $\alpha$ and $\beta$ with the strong composition $\gamma$ defined by 
\[
\gamma_i \coloneqq \sum_{t(j) = i} (\alpha\beta)_j,
\]
where $\alpha\beta$ denotes the concatenation of the two strong compositions.

\begin{example}\label{ex:overlapping}
We compute the overlapping shuffle product of $(2)$ and $(1,2)$:
\[
(2) \shuffle_o (1,2) = (2,1,2) + 2 \cdot (1,2,2) + (3,2) + (1,4).
\]
\end{example}

Although the relevant combinatorial construction was somewhat involved, the following positive multiplication formula is now essentially clear.

\begin{theorem}\label{thm:mult_M}
For strong compositions $\alpha$ and $\beta$, we have 
\[
M_\alpha \cdot M_\beta = \sum_\gamma c_{\alpha, \beta}^\gamma M_\gamma,
\]
where $c_{\alpha, \beta}^\gamma$ is the multiplicity of $\gamma$ in the overlapping shuffle product $\alpha \shuffle_o \beta$.
\end{theorem}

\begin{example}
To compute $M_{(2)} \cdot M_{(1,2)}$ via Theorem~\ref{thm:mult_M}, we compute the overlapping shuffle product of $(2)$ and $(1,2)$ as in Example~\ref{ex:overlapping}. Then, the coefficients on the various strong compositions give the coefficients on the various monomial quasisymmetric polynomials in the product:
\[
M_{(2)} \cdot M_{(1,2)} = M_{(2,1,2)} + 2 M_{(1,2,2)} + M_{(3,2)} + M_{(1,4)}.
\]
\end{example}

Given two strong compositions $\alpha$ and $\beta$, we say $\beta$ refines $\alpha$ and write $\beta \vDash \alpha$ if $\alpha$ can be obtained by summing \emph{consecutive} entries of $\beta$, e.g.\ $(1,2,1) \vDash (1,3)$ but $(2,1,1) \not \vDash (1,3)$.

Define the {\bf fundamental quasisymmetric polynomial} $F_\alpha$ by
\[F_\alpha(x_1,\ldots , x_n) \coloneqq \sum_{b}{\mathbf x}^b,\]
where the sum is over all distinct weak compositions $b$ with $b^+ \vDash \alpha$  and whose entries after position $n$ are all zero.

\begin{example}
We have \[F_{13}(x_1,x_2,x_3) = {\mathbf x}^{130} + {\mathbf x}^{103} + {\mathbf x}^{013} + {\mathbf x}^{112} + {\mathbf x}^{121}.\]
\end{example}

\begin{theorem}\label{thm:F2M}
The fundamental quasisymmetric polynomials expand positively in the monomial quasisymmetric polynomials:
\[F_\alpha(x_1,\ldots , x_n) = \sum_{\beta \vDash \alpha}M_\beta(x_1,\ldots , x_n).\]
\end{theorem}

A sign of the fundamental nature of the fundamental quasisymmetric polynomials is that they also have positive structure coefficients. Their multiplication is also governed by a shuffle product, that of S.~Eilenberg and S.~Mac Lane \cite{Eilenberg.MacLane}. Let $A$ and $B$ be words in the disjoint alphabets $\mathcal{A}$ and $\mathcal{B}$, respectively. A {\bf shuffle} of $A$ and $B$ is a permutation of the concatenation $AB$ such that the subword on the alphabet $\mathcal{A}$ is $A$ and the subword on $\mathcal{B}$ is $B$. Alternatively, if $A$ has length $m$ and $B$ has length $n$, we can think of a shuffle of $A$ and $B$ as a bijection 
\[
s : \{1, 2,  \dots, m+n \} \to \{1, 2, \dots, m+n\}
\] such that 
\[s(i) < s(j) \text{ whenever } i < j \leq m \text{ or } m < i < j.\] 

The shuffle product of two strong compositions $\alpha$ and $\beta$ is obtained as follows. Let $\mathcal{A}$ denote the alphabet of odd integers and let $\mathcal{B}$ denote the alphabet of even integers. Let $A$ be the word in $\mathcal{A}$ consisting of $\alpha_1$ copies of $2\ell(\alpha)-1$, followed by $\alpha_2$ copies of $2\ell(\alpha)-3$, all the way to $\alpha_{\ell(\alpha)}$ copies of $1$. Likewise, let $B$ denote the word in $\mathcal{B}$ consisting of $\beta_1$ copies of $2\ell(\beta)$, followed by $\beta_2$ copies of $2\ell(\beta)-2$, all the way to $\beta_{\ell(\beta)}$ copies of $2$. Let $\Sh(A,B)$ denote the set of the $|\alpha|+ |\beta| \choose |\beta|$ shuffles of $A$ and $B$. For each $C\in \Sh(A,B)$, let $\Des(C)$ denote the {\bf descent composition} of $C$, i.e.\ the strong composition obtained by decomposing $C$ into maximal runs of increasing entries and letting $\Des(C)_i$ be the number of entries in the $i$th increasing run of $C$. 
Finally, define the {\bf shuffle product} $\alpha \shuffle \beta$ of the strong compositions $\alpha$ and $\beta$ as the formal sum of strong compositions
\[
\alpha \shuffle \beta \coloneqq \sum_{C\in \Sh(A,B)}\Des(C). 
\]

\begin{example}\label{ex:shuffle}
Let $\alpha=(2)$ and $\beta=(1,2)$. Then $A=11$ and $B=422$. We compute the set of shuffles of $A$ and $B$:
\begin{align*}
\Sh(A,B) = \{ &4|22|11, 4|2|12|1, 4|122|1, 14|22|1, 4|2|112, \\
& 4|12|12, 14|2|12, 4|1122, 14|122, 114|22\},
\end{align*}
 where we have placed bars to indicate the decomposition of each shuffle into maximally increasing runs. The corresponding descent compositions are thus, respectively,
 \[
 \{(1,2,2), (1,1,2,1), (1,3,1), (2,2,1), (1,1,3), (1,2,2), (2,1,2), (1,4), (2,3), (3,2)\}.
 \]
 Hence, we have
\begin{align*}
(2)\shuffle (1,2) &= 2(1,2,2) + (1,1,2,1) + (1,3,1) + (2,2,1) \\
 &+ (1,1,3) + (2,1,2) + (1,4) + (2,3) + (3,2).
\end{align*}
Note that this sum is not multiplicity-free.
\end{example}

\begin{theorem}\label{thm:mult_F}
For strong compositions $\alpha$ and $\beta$, we have 
\[
F_\alpha \cdot F_\beta = \sum_\gamma c_{\alpha, \beta}^\gamma F_\gamma,
\]
where $c_{\alpha, \beta}^\gamma$ is the multiplicity of $\gamma$ in the ordinary shuffle product $\alpha \shuffle \beta$.
\end{theorem}

\begin{example}
To compute $F_{(2)} \cdot F_{(1,2)}$ via Theorem~\ref{thm:mult_F}, we compute the shuffle product of $(2)$ and $(1,2)$ as in Example~\ref{ex:shuffle}. Then, the coefficients on the various strong compositions give the coefficients on the various fundamental quasisymmetric polynomials in the product:
\[
F_{(2)} \cdot F_{(1,2)} = 2F_{(1,2,2)} + F_{(1,1,2,1)} + F_{(1,3,1)} + F_{(2,2,1)} + F_{(1,1,3)} + F_{(2,1,2)} + F_{(1,4)} + F_{(2,3)} + F_{(3,2)}.
\]
\end{example}

The final basis for $\qsym_n$ that we will consider is the basis of quasiSchur polynomials introduced in \cite{Haglund.Luoto.Mason.vanWilligenburg:quasiSchur}. For a detailed and readable survey of work related to this basis, see \cite{Luoto.Mykytiuk.vanWilligenburg}. For those unfamiliar with quasiSchur polynomials, the definition may appear strange and complicated; it originates as a particularly important and tractable piece of the theory of \emph{Macdonald polynomials}. It is not transparent from this definition that the quasiSchur polynomials are quasisymmetric, much less that they yield a basis of $\qsym_n$.

First, we must extend the definition of the Young diagram of a partition to a general weak composition $a = (a_1, a_2, \dots)$: Draw $a_i$ left-justified boxes in row $i$. (Here, in accordance with our English orientation on Young diagrams for partitions, row $1$ is the top row.) A {\bf (composition) tableau} of shape $a$ is an assignment of a positive integer to each box of the Young diagram for $a$. (Sometimes, we will augment such a tableau with an extra column $0$ of boxes on the left side (the {\bf basement}) and write $b_i$ for the positive integer labeling the basement box in row $i$.)

A {\bf triple} of boxes in a composition tableau $T$ is a set of three boxes in one of the two following configurations:
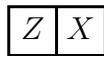
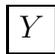
\begin{align*}
&\ytableaushort{Z X, \none {\none[\vdots]}, \none Y}  &\ytableaushort{Y, {\none[\vdots]}, Z X} \\
&\text{upper row weakly longer}  &\text{upper row strictly shorter}
\end{align*}
Note, in particular, that a triple has exactly two boxes sharing a row and exactly two boxes sharing a column.
We say a triple is {\bf inversion} if it is \emph{not} the case that its labels satisfy $X \leq Y \leq Z$.

A composition tableau is {\bf semistandard} if 
\begin{enumerate}
\item[(S.1)] entries do not repeat in a column,
\item[(S.2)] rows weakly decrease from left to right,
\item[(S.3)] every triple is inversion,
\item[(S.4)] entries in the first column equal their row indices. 
\end{enumerate} 
(Note that, in the case of partition shape, this definition unfortunately does not coincide with the definition of semistandard tableaux we have given previously.)
For a weak composition $a$, let $\ASSF(a)$ denote the set of semistandard tableaux of shape $a$.
The {\bf quasiSchur polynomial} for the strong composition $\alpha$ is then given by
\begin{equation}\label{eq:quasiSchur}
\qschur_\alpha(x_1, \ldots , x_n) = \sum_{a^+ = \alpha} \sum_{T \in \ASSF(a)} \xx^{\wt(T)},
\end{equation}
where the first sum is over all weak compositions $a$ of length $n$ with positive part $\alpha$.

\begin{example}
For $\alpha=(1,3)$ and $n=3$, we have
\[\qschur_{(1,3)}(x_1,x_2,x_3) = \xx^{130}+\xx^{220} + \xx^{103} + \xx^{202} + 2\xx^{112} + \xx^{121} + \xx^{211} + \xx^{013} + \xx^{022},\]
where the monomials are determined by the semistandard composition tableaux shown in Figure~\ref{fig:ASSF}.
\end{example}

\begin{figure}[h]
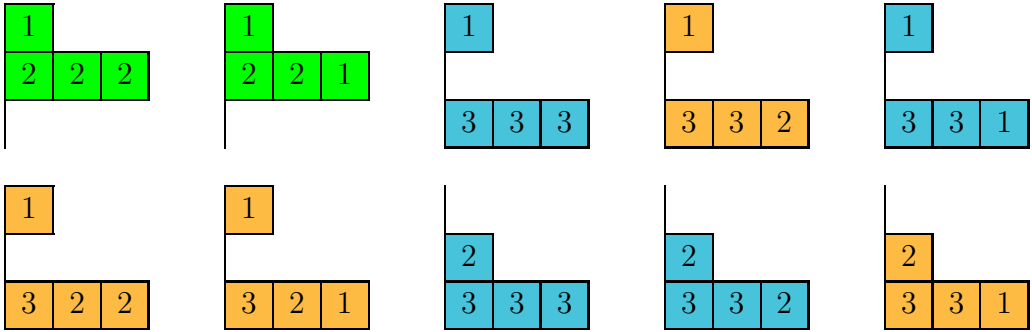

\begin{displaymath}
\begin{array}{cccccc}
\begin{ytableau}[*(green)] \none & 1 \\ \none  & 2 & 2 & 2  \\ \none \vline   \end{ytableau} 
& \begin{ytableau}[*(green)]  \none & 1 \\ \none  & 2 & 2 & 1  \\ \none \vline   \end{ytableau} 
& \begin{ytableau}[*(SkyBlue)] \none & 1 \\ \none \vline \\ \none  & 3 & 3 & 3     \end{ytableau}
& \begin{ytableau}[*(Dandelion)] \none & 1 \\ \none \vline \\ \none  & 3 & 3 & 2     \end{ytableau}
& \begin{ytableau}[*(SkyBlue)] \none & 1 \\ \none \vline \\ \none  & 3 & 3 & 1     \end{ytableau} \\ \\
 \begin{ytableau}[*(Dandelion)]  \none & 1 \\ \none \vline \\ \none  & 3 & 2 & 2     \end{ytableau}
& \begin{ytableau}[*(Dandelion)]  \none & 1 \\ \none \vline \\ \none  & 3 & 2 & 1     \end{ytableau}
& \begin{ytableau}[*(SkyBlue)]  \none \vline \\   \none & 2  \\ \none  & 3 & 3 & 3 \end{ytableau}
& \begin{ytableau}[*(SkyBlue)]  \none \vline \\   \none & 2  \\ \none  & 3 & 3 & 2 \end{ytableau}
& \begin{ytableau}[*(Dandelion)]   \none \vline \\   \none & 2  \\ \none  & 3 & 3 & 1 \end{ytableau}
\end{array}
\end{displaymath}
\caption{The $10$ semistandard composition tableaux associated to the quasiSchur polynomial $\qschur_{(1,3)}(x_1,x_2,x_3)$. The quasiYamanouchi tableaux are shaded in \textcolor{SkyBlue}{blue}, the initial tableaux in \textcolor{Dandelion}{orange}, and those that are both quasiYamanouchi and initial in \textcolor{green}{green}.}\label{fig:ASSF}
\end{figure}

From the given definition of quasiSchur polynomials, it is not clear that they are natural objects that we should expect to exhibit any nice properties. Nonetheless, they participate in two beautiful positive combinatorial rules for change of basis.

Since $\sym_n \subset \qsym_n$, we can ask how bases of $\sym_n$ expand in bases of $\qsym_n$. First, observe the following straightforward formula for the $M_\alpha$-expansion of the monomial symmetric polynomial $m_\lambda$. For a weak composition $a$, we write $\sort{a}$ for the partition formed by sorting the entries of $a$ into weakly decreasing order.

\begin{proposition}\label{prop:m2M}
The monomial symmetric polynomials expand positively in the monomial quasisymmetric polynomials:
\[ m_\lambda(x_1,\ldots , x_n) = \sum_{\sort{\alpha} = \lambda} M_\alpha(x_1,\ldots , x_n).\]
\end{proposition}

The quasiSchur expansion of a Schur polynomial is beautifully parallel to the formula of Proposition~\ref{prop:m2M}.

\begin{theorem}[\cite{Haglund.Luoto.Mason.vanWilligenburg:quasiSchur}]\label{thm:s2S}
The Schur polynomials expand positively in the quasiSchur polynomials:
\[ s_\lambda(x_1,\ldots , x_n) = \sum_{\sort{\alpha} = \lambda} S_\alpha(x_1,\ldots , x_n).\]
\end{theorem}

\begin{remark}
Considering Figure~\ref{fig:qsym} together with Proposition~\ref{prop:m2M} and Theorem~\ref{thm:s2S}, one might be tempted to define polynomials
\[ f_\lambda(x_1,\ldots , x_n) = \sum_{\sort{\alpha} = \lambda} F_\alpha(x_1,\ldots , x_n).\]
Extrapolating from Figure~\ref{fig:qsym}, it might appear plausible that $\{ f_\lambda \}$ should form a basis of $\sym_n$, perhaps even with positive structure coefficients. However, the polynomials $f_\lambda$ are in general not even symmetric!

For example, in four or more variables, we have by Theorem~\ref{thm:F2M} and Proposition~\ref{prop:m2M} that
\begin{align*}
f_{31} &= F_{31} + F_{13} \\
&= M_{31} + M_{211} + 2M_{121} + M_{112} + M_{13} + 2M_{1111} \\
&= m_{31} + m_{211} + 2 m_{1111} + M_{121},
\end{align*}
a symmetric polynomial plus $M_{121}$.
\end{remark}

To describe the expansion of quasiSchur polynomials into the fundamental basis, we isolate an important subclass of semistandard composition tableaux. Fix a strong composition $\alpha$ and consider $T \in \ASSF(a)$ for some $a$ with $a^+ = \alpha$. We say that $T$ is {\bf quasiYamanouchi} if for every integer $i$ appearing in $T$, either
\begin{itemize}
\item an $i$ appears in the first column, or
\item there is an $i+1$ weakly right of an $i$.
\end{itemize}
We say $T$ is {\bf initial} if the set of integers $i$ appearing in $T$ is an initial segment of $\mathbb{Z}_{>0}$.

\begin{theorem}\label{thm:S2F}
The quasiSchur polynomials expand positively in the fundamental quasisymmetric polynomials:
\[ S_\alpha(x_1,\ldots , x_n) = \sum_{T} F_{\wt(T)}(x_1,\ldots , x_n),\]
where the sum is over all initial quasiYamanouchi tableaux $T$ such that $T \in \ASSF(a)$ for some $a$ with $a^+ = \alpha$.
\end{theorem}

A positive formula for the expansion of quasiSchur polynomials in fundamental quasisymmetric polynomials was first given in \cite{Haglund.Luoto.Mason.vanWilligenburg:quasiSchur} in terms of \emph{standard augmented fillings}. The formula in Theorem~\ref{thm:S2F} above follows as a consequence of a result in \cite{Searles}; we state the expansion in these terms for consistency with formulas in the upcoming sections.

\begin{remark}
Unlike the other two bases of $\qsym_n$ that we have considered, the quasiSchur basis does not have positive structure coefficients. For an example, see \cite[$\mathsection 7.1$]{Haglund.Luoto.Mason.vanWilligenburg:quasiSchur}. However, \cite{Haglund.Luoto.Mason.vanWilligenburg:refinement} proves a slightly weaker form of positivity, giving a positive combinatorial formula for the quasiSchur expansion of the product of a quasiSchur polynomial by a \emph{Schur} polynomial.
\end{remark}

Just as the combinatorics of $\sym_n$ is related to the representation theory of symmetric groups, the combinatorics of $\qsym_n$ turns out to be related to the representation theory of \emph{$0$-Hecke algebras} (in type $A$). First, let us recall the standard Coxeter presentation of the symmetric group $\mathcal{S}_n$. It is easy to see that $\mathcal{S}_n$ is generated by the simple transpositions $s_i \coloneqq (i \; i+1)$ for $1 \leq i < n$. With more effort, one establishes that a generating set of relations is given by
\begin{itemize}
\item $s_i^2 = \id$,
\item $s_i s_j = s_j s_i$ for $|i-j| > 1$, and
\item  $s_i s_{i+1} s_i = s_{i+1} s_i s_{i+1}$.
\end{itemize}

The $0$-Hecke algebra $\mathcal{H}_n$ is the unital associative algebra over $\mathbb{C}$ defined by a very similar presentation: $\mathcal{H}_n$ is generated by symbols $\sigma_i$ (for $1 \leq i < n$) subject to
\begin{itemize}
\item $\sigma_i^2 = \sigma_i$,
\item $\sigma_i \sigma_j = \sigma_j \sigma_i$ for $|i-j| > 1$, and
\item  $\sigma_i \sigma_{i+1} \sigma_i = \sigma_{i+1} \sigma_i \sigma_{i+1}$.
\end{itemize}
That is, the $\{ \sigma_i \}$ in $\mathcal{H}_n$ act exactly like the corresponding $\{ s_i \}$ in $\mathcal{S}_n$, except that they are idempotent instead of being involutions.

The representation theory of $\mathcal{H}_n$ was first worked out in detail by P.~Norton \cite{Norton}. Despite the similarly between the descriptions of $\mathcal{S}_n$ and $\mathcal{H}_n$, their representation theory is rather different, as $\mathcal{H}_n$ is not semisimple. Indeed, the irreducible representations of $\mathcal{H}_n$ are all $1$-dimensional, while $\mathcal{S}_n$ has irreducible representations of higher dimension. The irreducible representations of $\mathcal{H}_n$ are equinumerous with the set of compositions $\alpha \vDash (n)$. 

There is a \emph{quasisymmetric Frobenius character map} \cite{Duchamp.Krob.Leclerc.Thibon, Krob.Thibon} taking $0$-Hecke-representations to quasisymmetric functions in such a way that the irreducible representations map to the fundamental quasisymmetric functions $F_\alpha$. In this way, if the quasisymmetric function $f$ corresponds to the representation $M$, then decomposing $f$ as a sum of fundamental quasisymmetric functions corresponds to identifying the unique direct sum of irreducible $0$-Hecke representations that is equivalent to $M$ in the Grothendieck group of finite-dimensional representations. Certain explicit and combinatorial $\mathcal{H}_n$-representations are known whose quasisymmetric Frobenius characters are precisely the quasiSchur functions \cite{Tewari.vanWilligenburg}; unfortunately, these representations are not generally indecomposable.

The geometry of $\qsym_n$ is much less well understood. In addition to its obvious product structure, $\qsym$ also possesses a compatible coproduct, turning it into a \emph{Hopf algebra}. Although we won't describe it here, there is an important Hopf algebra ${\rm NSym}$ of \emph{noncommutative symmetric functions} that is Hopf-dual to $\qsym$. It is surprisingly easy to see that ${\rm NSym}$ is isomorphic to the homology of the loop space of the suspension of $\mathbb{CP}^\infty$, where the product structure on $H_\star(\Omega \Sigma \mathbb{CP}^\infty)$ is given by concatenation of loops \cite{Baker.Richter}. Indeed, this isomorphism even holds on the level of Hopf algebras. Since $\Omega \Sigma\mathbb{CP}^\infty$ is an H-space, its homology and cohomology are dual Hopf algebras. (See, for example, \cite{Hatcher,Whitehead} for background on H-spaces and Hopf algebras.) From this fact and the fact that $\qsym$ is Hopf-dual to ${\rm NSym}$, it follows that $H^\star(\Omega \Sigma \mathbb{CP}^\infty)$ is isomorphic to $\qsym$. This interpretation was used in \cite{Baker.Richter} to give cohomological proofs of various properties of $\qsym$; however, it seems that much more could be done from this perspective.
A recent construction, which appears closely related, identifes $\qsym$ with the Chow ring of an algebraic stack of \emph{expanded pairs} \cite{Oesinghaus}.

\section{The asymmetric world}\label{sec:poly}
In this section, we consider our fourth and final action of $\mathcal{S}_n$ on $\poly_n$, the trivial action. Although the action involved is the silliest possible one, the associated combinatorics is not at all silly, full of rich internal structure and deep connections to geometry and representation theory. The invariant ring of this trivial action is, of course, the ring $\poly_n$ itself. However, to emphasize analogies with the previous two sections, we will think of $\poly_n$ in this context as the {\bf ring of asymmetric functions} $\asym_n$.

Bases of $\asym_n$ are indexed by weak compositions $a$ of length at most $n$, with the most obvious basis of $\asym_n$ being given by individual monomials:
\[
\monomial_a \coloneqq \xx^a.
\]
Just as $\{ m_\lambda\}$ is not the most interesting basis of $\sym_n$, the $\{\monomial_a\}$ basis of $\asym_n$ is not very interesting either! We will explore here seven additional bases of rather less trivial nature.

\begin{figure}[h]
\begin{tikzpicture}
\node[plaque,posplaque] (Schubert) {\pcolor{$\schub_a$}};
\node[right=1 of Schubert,plaque] (Demazure) {\pcolor{$\key_a$}};
\node[right=1 of Demazure,plaque] (qkey) {\pcolor{$\qkey_a$}};
\node[above right=0.5 and 1 of qkey,plaque] (atom) {\pcolor{$\atom_a$}};
\node[below right=0.5 and 1 of qkey,posplaque] (slide) {\pcolor{$\slide_a$}};
\node[below right = 0.5 and 1 of slide,posplaque] (mslide) {\pcolor{$\mslide_a$}};
\node[below right =0.5 and 1 of atom,plaque] (particle) {\pcolor{$\particle_a$}};
\node[below right = 0.5 and 1 of particle,posplaque] (monomial) {\pcolor{$\monomial_a$}};
  \begin{scope}[nodes = {draw = none}]
    \path (Schubert) edge[posex]   (Demazure)
    (Demazure) edge[posex]  (qkey)
    (qkey) edge[posex] (slide)
    (qkey) edge[posex] (atom)
    (slide) edge[posex] (mslide)
    (slide) edge[posex] (particle)
    (atom) edge[posex] (particle)
    (particle) edge[posex] (monomial)
    (mslide) edge[posex] (monomial)
      ;
  \end{scope}
\end{tikzpicture}
\caption{The eight bases of $\asym_n$ considered here. The arrows denote that the basis at the head refines the basis at the tail. The star-shaped nodes have positive structure coefficients.}\label{fig:poly}
\end{figure}
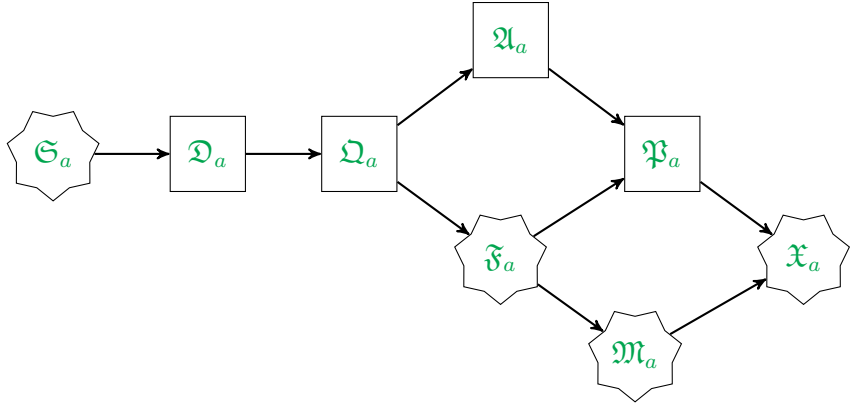

Arguably, the most interesting basis of $\asym_n$ is given by the \emph{Schubert polynomials} of A.~Lascoux and M.-P.~Sch\"utzenberger \cite{Lascoux.Schutzenberger}. Instead of indexing Schubert polynomials by weak compositions, it is more convenient to index them by permutations. Hence, we first recall a standard way to translate between permutations and weak compositions. For a permutation $\pi  \in \mathcal{S}_n$, let $a_i$ denote the number of integers $j > i$ such that $w(i) > w(j)$. (Note that $a_i \leq n-i$.) The weak composition $a_\pi = (a_1, a_2, \dots, a_n)$ is called the {\bf Lehmer code} of $\pi$. Visually, one may determine the Lehmer code of a permutation $\pi$ as follows.

Consider an $n \times n$ grid of boxes and place a laser gun (or dot) in each position $(i, \pi(i))$. Each laser gun fires to the right and down, destroying all boxes directly to its right and all boxes directly below itself (including its own box). The surviving boxes are the {\bf Rothe diagram} $RD(\pi)$ of the permutation $\pi$. One checks that the Lehmer code of $\pi$ records the number of boxes in each row of $RD(\pi)$.

\begin{example}
Let $\pi=2413$. The Rothe diagram $RD(\pi)$ is shown below
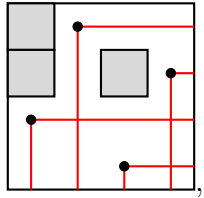
\[
 \begin{tikzpicture}[x=1.5em,y=1.5em]
      \draw[color=black, thick](0,1)rectangle(4,5);
     \filldraw[color=black, fill=gray!30, thick](0,4)rectangle(1,5);
     \filldraw[color=black, fill=gray!30, thick](0,3)rectangle(1,4);
     \filldraw[color=black, fill=gray!30, thick](2,3)rectangle(3,4);
     \draw[thick,red] (1.5,4.5)--(1.5,1);
      \draw[thick,red] (1.5,4.5)--(4,4.5);
      \draw[thick,red] (3.5,3.5)--(3.5,1);
      \draw[thick,red] (3.5,3.5)--(4,3.5);
       \draw[thick,red] (0.5,2.5)--(0.5,1);
      \draw[thick,red] (0.5,2.5)--(4,2.5);
      \draw[thick,red] (2.5,1.5)--(2.5,1);
      \draw[thick,red] (2.5,1.5)--(4,1.5);
           \filldraw [black](1.5,4.5)circle(.1);
     \filldraw [black](3.5,3.5)circle(.1);
     \filldraw [black](0.5,2.5)circle(.1);
     \filldraw [black](2.5,1.5)circle(.1);
     \end{tikzpicture},
     \]
   where the surviving boxes are shaded in grey. Hence, the Lehmer code of $\pi = 2413$ is $(1,2,0,0)$.
\end{example}

Consider the action of $\mathcal{S}_n$ on $\poly_n$ from Section~\ref{sec:sym}, where permutations act by permuting variables. Now, for each positive integer, define an operator $\partial_i$ on $\poly_n$ by
\[
\partial_i(f) \coloneqq \frac{f - (i \; i+1) \cdot f}{x_i - x_{i+1}}.
\]
Note that $\partial_i(f)$ is symmetric in the variables $x_i$ and $x_{i+1}$. Now, for each permutation $w$ of the form $n (n-1) \cdots 321$ (in one-line notation), the {\bf Schubert polynomial} $\schub_w$ is defined to be
\begin{equation}
\schub_w \coloneqq \prod_{i = 1}^n x_i^{n-i} = \monomial_{(n-1, n-2, \dots,1, 0)}.\label{eq:schub_w0}
\end{equation}
(These permutations are exactly those that are longest in Coxeter length in $\mathcal{S}_n$ for some $n$.) For other permutations $w$, the corresponding Schubert polynomials are defined recursively by 
\[
\schub_w \coloneqq \partial_i \schub_{w (i \; i+1)},
\]
for any $i$ such that $w(i) < w(i+1)$. Amazingly, this recursive definition is self-consistent, so there is a uniquely defined Schubert polynomial $\schub_w$ for each permutation $w$.

Our first task in this section will be to obtain a more concrete understanding of Schubert polynomials by describing how to write them non-recursively in the monomial basis $\{ \monomial_a \}$. We'll describe three different combinatorial formulas for this expansion, exploring some other families of polynomials along the way. 

The first such formula to be proven was given by S.~Billey, W.~Jockusch, and R.~Stanley \cite{Billey.Jockusch.Stanley}. For a permutation $\pi$, a {\bf reduced factorization} of $\pi$ is a way of writing $\pi$ as a product $s_{i_1} s_{i_2} \cdots s_{i_k}$ of simple transpositions with $k$ as small as possible. The sequence of subscripts $i_1 i_2 \cdots i_k$ is called a {\bf reduced word} for $\pi$. We write $\reduced(\pi)$ for the set of all reduced words of the permutation $\pi$. Note that every reduced word $\alpha$ is a strong composition. Given two strong compositions $\alpha$ and $\beta$, we say that $\beta$ is {\bf $\alpha$-compatible} if 
\begin{enumerate}
\item[(R.1)] $\alpha$ and $\beta$ have the same length,
\item[(R.2)] $\beta$ is weakly increasing (i.e., $\beta_i \leq \beta_j$ for $i < j$),
\item[(R.3)] $\beta$ is bounded above by $\alpha$ (i.e., $\beta_i \leq \alpha_i$ for all $i$), and
\item[(R.4)] $\beta$ strictly increases whenever $\alpha$ does (i.e., if $\alpha_i < \alpha_{i+1}$, then $\beta_i < \beta_{i+1}$).
\end{enumerate} In this case, we write $\beta \looparrowright \alpha$.

\begin{example}\label{ex:compatible_sequence}
If $\alpha$ is the strong composition $121$ (a reduced word for the longest permutation in $\mathcal{S}_3$), then no strong composition is $\alpha$-compatible. For suppose $\beta$ were $\alpha$-compatible. Since $\alpha_1 < \alpha_2$, we must have $\beta_1 < \beta_2$ by (R.4). Hence, $\beta_2 \geq 2$. Therefore, by (R.2), $\beta_3 \geq 2$. But this is incompatible with (R.3), since $\alpha_3 = 1$.

On the other hand, for $\gamma = 212$ (the other reduced word for this permutation), there is exactly one $\gamma$-compatible strong composition $\delta$. By (R.3), we have $\delta_2 = 1$, and hence by (R.2) we also have $\delta_1 = 1$. By (R.4), $\delta_3 > \delta_2 = 1$, but by (R.3) $\delta_3 \leq 2$. Hence, $\delta = 112$ is the only $\gamma$-compatible strong composition. We write $112 \looparrowright 212$.
\end{example}

\begin{theorem}[{\cite[Theorem~1.1]{Billey.Jockusch.Stanley}}]\label{thm:schub2X}
The Schubert polynomials expand positively in monomials:
\[
\schub_\pi = \sum_{\alpha \in \reduced(\pi)} \sum_{\beta \looparrowright \alpha} \prod_i x_{\beta_i}.
\] 
\end{theorem}

\begin{example}
Let $\pi = 321 =  s_1 s_2 s_1 = s_2 s_1 s_2$ be the longest permutation in $\mathcal{S}_3$. Then, by Example~\ref{ex:compatible_sequence}, we have 
\begin{align*}
\schub_\pi &= \sum_{\alpha \in \reduced(\pi)} \sum_{\beta \looparrowright \alpha} \prod_i x_{\beta_i} \\
&= \sum_{\beta \looparrowright 121} \prod_i x_{\beta_i} +  \sum_{\delta \looparrowright 212} \prod_i x_{\delta_i} \\
&= 0 + x_1 x_1 x_ 2 = x_1^2 x_2 = \monomial_{(2,1,0)}.
\end{align*}
Note that this calculation is consistent with the definition given in Equation~\eqref{eq:schub_w0}.
\end{example}

It might be reasonable to expect an important basis of $\asym_n$ to restrict to an important basis of the subspace $\sym_n \subset \asym_n$. Indeed, one piece of evidence for the importance of Schubert polynomials is that those Schubert polynomials lying inside $\sym_n$ are exactly its basis of Schur polynomials. That is, every Schur polynomial is a Schubert polynomial and the Schubert basis of $\asym_n$ is a lift of the Schur basis of $\sym_n$. 

To realize the Schur polynomial $s_\lambda \in \sym_n$ as a Schubert polynomial, first realize the partition $\lambda$ as a weak composition of length $n$ by padding it by an appropriate number of final $0$s. Now, reverse the letters of $\lambda$, so it becomes a weakly increasing sequence. The resulting weak composition is the Lehmer code of a unique permutation $\pi_\lambda$, and one has $s_\lambda = \schub_{\pi_\lambda}$. Equivalently, for $i \leq n$ one has $\pi_\lambda(i) = \lambda_{n-i + 1} + i$ and for $i > n$ one has $\pi_\lambda(i) = \min \left( \mathbb{Z}_{> 0}  \setminus \{ \pi_\lambda(j) : j < i   \} \right)$.

Since the Schubert polynomials lift the Schur polynomials, one might wonder whether the Littlewood-Richardson rule (Theorem~\ref{thm:LRrule}) also lifts to a positive combinatorial rule for the structure coefficients of the Schubert basis. Indeed, the Schubert basis of $\asym_n$, like the Schur basis of $\sym_n$, has positive structure coefficients! However, no combinatorial proof of this fact is known and we lack any sort of positive combinatorial rule (even conjectural) to describe these structure coefficients (except in a few very special cases, such as when the Schubert polynomials are actually Schur polynomials). Discovering and proving such a rule is one of the most important open problems in algebraic combinatorics. Part of our motivation for studying the combinatorial theory of $\asym_n$ is the hope that such a theory will eventually lead to a Schubert structure coefficient rule, just as the Littlewood-Richardson rule for Schur polynomial structure coefficients eventually developed from the combinatorial theory of $\sym_n$.

Without such a combinatorial rule, how then do we know that the Schubert basis has positive structure coefficients? The answer comes, once again, from geometry and from representation theory. Geometrically, instead of looking at a complex Grassmannian, as we did for Schur polynomials, we should consider a complex flag variety $\mathrm{Flags}_n$, the classifying space for complete flags $V_0 \subset V_1 \subset \cdots \subset V_n$ of nested complex vector bundles with $V_k$ of rank $k$. This space has an analogous cell decomposition by Schubert varieties, yielding an effective basis of the Chow ring. By identifying Schubert varieties with corresponding Schubert polynomials, multiplying Schubert polynomials corresponds to the intersection product on Schubert varieties and positivity of structure coefficients follows. An alternative proof of positivity \cite{Watanabe:approach,Watanabe:tensor} is given by interpreting Schubert polynomials as characters of certain \emph{KP-modules} (introduced in \cite{Kraskiewicz.Pragacz:1,Kraskiewicz.Pragacz:2}) for Borel Lie algebras.

The formula of Theorem~\ref{thm:schub2X} naturally leads us to consider another family of polynomials. Suppose we fix a reduced word $\alpha \in \reduced(\pi)$ for some $\pi \in \mathcal{S}_n$. Then, Theorem~\ref{thm:schub2X} suggests defining a polynomial
\begin{equation}\label{eq:slide_by_reduced_word}
\slide(\alpha) \coloneqq \sum_{\beta \looparrowright \alpha} \prod_i x_{\beta_i},
\end{equation}
so that Theorem~\ref{thm:schub2X} may be rewritten as
\[
\schub_\pi = \sum_{\alpha \in \reduced(\pi)} \slide(\alpha).
\]
Indeed, the formula of Equation~\eqref{eq:slide_by_reduced_word} makes sense for any strong composition $\alpha$, not necessarily a reduced word of a permutation. 

Labeling these polynomials by strong compositions $\alpha$ is unnatural for at least two reasons. For some $\alpha$, we have $\slide(\alpha) = 0$; for example, we have $\slide(1,2,1) = 0$ by Example~\ref{ex:compatible_sequence}. Those $\slide(\alpha)$ that are nonzero are called the {\bf fundamental slide polynomials}; these were introduced in \cite{Assaf.Searles}, although the alternate definition we give here follows \cite{Assaf:Stanley}. Also, for $\alpha \neq \alpha'$, we can have $\slide(\alpha) = \slide(\alpha') \neq 0$; for example, by Example~\ref{ex:compatible_sequence} we have $\slide(212) = \monomial_{(2,1,0)}$, but it is also equally clear that $\slide(312) = \monomial_{(2,1,0)}$.

For any strong composition $\alpha$, note that, if $\alpha$ has any compatible sequences, then it has a unique such compatible sequence $\beta(\alpha)$ that is termwise maximal. Let $a(\alpha)_i$ denote the multiplicity of $i$ in $\beta(\alpha)$. Then, we define
\[
\slide_{a(\alpha)} \coloneqq \slide(\alpha).
\]
It is clear then that every fundamental slide polynomial is, in this fashion, uniquely indexed by a weak composition $a$. Moreover, every weak composition $a$ appears as an index on some $\slide_a$, and we have $\slide_a \neq \slide_b$ if $a \neq b$. It is then not hard to see by triangularity in the $\monomial_a$ basis that the set of fundamental slide polynomials forms another basis of $\poly_n$.

Clearly, the fundamental slide polynomials expand positively in the monomial basis $\{ \monomial_a \}$. It is useful to have a formula for this expansion of $\slide_a$, based only on the weak composition $a$. We first need a partial order on weak compositions: we write $a \geq b$ and say $a$ {\bf dominates} $b$ if we have
\[
\sum_{i=1}^k a_i \geq \sum_{i=1}^k b_i
\]
for all $k$. (Note that the restriction of this partial order to the set of partitions recovers the usual notion of dominance order.)

\begin{theorem}[\cite{Assaf:Stanley,Assaf.Searles}]\label{thm:slide2X}
The fundamental slide polynomials expand positively in monomials:
\[
\slide_a = \sum_{\substack{b \geq a \\ b^+ \vDash a^+ }} \monomial_b
\]
\end{theorem}

Essentially by definition, the fundamental slide polynomials are pieces of Schubert polynomials. Although the basis of Schubert polynomials has positive structure constants, there is no reason to expect this property to descend to this basis of pieces. Remarkably, however, the fundamental slide polynomials also have positive structure constants! The first clue that this might be the case comes from considering the intersection of the fundamental slide basis with the subring $\qsym_n \subset \poly_n$.

\begin{theorem}[\cite{Assaf.Searles}]\label{thm:slide_lift_fund}
We have $\slide_a \in \qsym_n$ if and only if $a$ is of the form $0^k \alpha$, where $\alpha$ is a strong composition of length $n-k$ and $0^k \alpha$ denotes the weak composition obtained from $\alpha$ by prepending $k$ $0$s.

Moreover, we have $\slide_{0^k\alpha} = \fundamental_\alpha$. Thus, the fundamental slide basis of $\poly_n$ is a lift of the fundamental quasisymmetric polynomial basis of $\qsym_n$.
\end{theorem}

In light of Theorem~\ref{thm:slide_lift_fund}, one might hope to extend the combinatorial multiplication rule of Theorem~\ref{thm:mult_F} to fundamental slide polynomials. Indeed, this is possible. We need to extend the notion of the shuffle product of two strong compositions from Section~\ref{sec:qsym} to the \emph{slide product} or pairs of weak compositions $a,b$. Here, we borrow notation from \cite{Pechenik.Searles}. As before, let $\mathcal{A}$ denote the alphabet of odd integers and let $\mathcal{B}$ denote the alphabet of even integers. Let $A$ be the word in $\mathcal{A}$ consisting of $a_1$ copies of $2\ell(a)-1$, followed by $a_2$ copies of $2\ell(a)-3$, all the way to $a_{\ell(a)}$ copies of $1$. Likewise, let $B$ denote the word in $\mathcal{B}$ consisting of $b_1$ copies of $2\ell(b)$, followed by $b_2$ copies of $2\ell(b)-2$, all the way to $b_{\ell(b)}$ copies of $2$. 

For any word $W$ in a totally ordered alphabet $\mathcal{Z}$, let $\Runs(W)$ denote the sequence of successive maximally increasing runs of letters of $W$ read from left to right. For a sequence $S$ of words in $\mathcal{Z}$ and any subalphabet $\mathcal{Y} \subseteq \mathcal{Z}$, write $\Comp_\mathcal{Y}(S)$ for the weak composition whose $i$th coordinate is the number of letters of $\mathcal{Y}$ in the $i$th word of $S$.

Let $\Sh(a,b)$ denote the set of those shuffles $C$ of $A$ and $B$ such that \[
\Comp_\mathcal{A}(\Runs(C)) \geq a  \text{ and } \Comp_\mathcal{B}(\Runs(C)) \geq b.\] 
For $C \in \Sh(a,b)$, let $\Bumpruns(C)$ denote the unique dominance-minimal way to insert words of length $0$
into $\Runs(C)$ while preserving $\Comp_\mathcal{A}(\Bumpruns(C)) \geq a$ and $\Comp_\mathcal{B}(\Bumpruns(C)) \geq b$.
Finally, define the {\bf slide product} $a \shuffle b$ of the weak compositions $a$ and $b$ as the formal sum of weak compositions
\[
a \shuffle b \coloneqq \sum_{C\in \Sh(a,b)} \Comp_\mathbb{Z}(\Bumpruns(C)). 
\]

\begin{example}\label{ex:mult_slide}
Let $a = (0,1,0,2)$ and $b = (1,0,0,1)$. Then we consider the words $A = 511$ and $B = 82$. The set of all shuffles of $A$ and $B$ is \[
\{ 51182, 51812,58112,85112, 51821,58121,85121, 58211,85211,82511 \}.
\]
Many of these shuffles $C$ fail $\Comp_\mathcal{B}(\Runs(C)) \geq b$; for example, with $C = 51821$, we have $\Runs(C) = (5,18,2,1)$ and hence $\Comp_\mathcal{B}(\Runs(C)) = (0,1,1,0) \not \geq b$.
Thus we have \[
\Sh(a,b) = \{  58112,85112, 58121,85121, 58211,85211,82511 \}.
\]
The corresponding $\Bumpruns(C)$ for $C \in \Sh(a,b)$ are
\[
\{ 
(58, \epsilon, \epsilon, 112), (8,5, \epsilon, 112), (58, \epsilon, 12,1), (8,5,12,1), (58, \epsilon, 2,11), (8,5,2,11), (8,25, \epsilon, 11)
\},
\]
where $\epsilon$ denotes the empty word.
 Thus, we have
\begin{align*}
(0,1,0,2) \shuffle (1,0,0,1) &= (2,0,0,3) + (1,1,0,3) + (2,0,2,1) + (1,1,2,1) + (2,0,1,2) \\
 &+ (1,1,1,2) + (1,2,0,2).
\end{align*}
\end{example}

Finally, we can state the multiplication rule for fundamental slide polynomials.

\begin{theorem}[\cite{Assaf.Searles}]\label{thm:mult_slide}
For weak compositions $a$ and $b$, we have 
\[
\slide_a \cdot \slide_b = \sum_c C_{a, b}^c \slide_c,
\]
where $C_{a, b}^c$ is the multiplicity of $c$ in the slide product $a \shuffle b$.
\end{theorem}

It seems reasonable to expect that the positivity of Theorem~\ref{thm:mult_slide} reflects some geometry or representation theory governed by fundamental slide polynomials. Sadly, no such interpretation of fundamental slide polynomials is known, except in the quasisymmetric case.

A second formula for Schubert polynomials uses the combinatorial model of pipe dreams. This model was successively developed in \cite{Bergeron.Billey, Fomin.Kirillov, Knutson.Miller}. A {\bf pipe dream} $P$ is a tiling of the grid of boxes (extending infinitely to the east and south) with {\bf turning pipes} $\turn$ and finitely many {\bf crossing pipes} $\cross$. Such a tiling gives rise to a collection of lines called {\bf pipes}, which one imagines traveling from the left side of the grid (the negative $y$-axis) to the top side (the positive $x$-axis). A pipe dream $P$ is called {\bf reduced} if no two pipes cross each other more than once. The permutation corresponding to a reduced pipe dream is the permutation given (in one-line notation) by the columns in which the pipes end. The {\bf weight} $\wt(P)$ of a pipe dream is the weak composition whose $i$th entry is the number of crossing pipe tiles in row $i$ of $P$ (where row $1$ is the top row).

\begin{example}
The pipe dream 
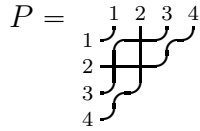
\[
P = \pipes{ & 1 & 2 & 3 & 4 \\ 
1 & \turn & \cross & \turn & \tail \\
2 & \cross &  \cross & \tail \\
3 &\turn & \tail \\
4 & \tail }
\]
corresponds to the permutation $1432$. (Here, we omit the infinite collection of $\turn$ tiles extending uninterestingly to the southeast.) The pipe dream $P$ has weight $(1,2,0)$.
\end{example}

Given a permutation $\pi$, let $\PD(\pi)$ denote the set of reduced pipe dreams for $\pi$.

\begin{theorem}\cite{Bergeron.Billey, Billey.Jockusch.Stanley}
The Schubert polynomial $\schub_\pi$ is the generating function of reduced pipe dreams for $\pi$, i.e.,
\[\schub_\pi = \sum_{P\in \PD(\pi)}\xx^{\wt(P)}.\]
\end{theorem}

\begin{example}\label{ex:pipes}
We have $\schub_{15324} = \xx^{031} + \xx^{121} +\xx^{211} +\xx^{310} +\xx^{310} +\xx^{130} +\xx^{220}$, 
where the monomials are determined by the pipe dreams shown in Figure~\ref{fig:pipes}.
\end{example}

\begin{figure}[h]
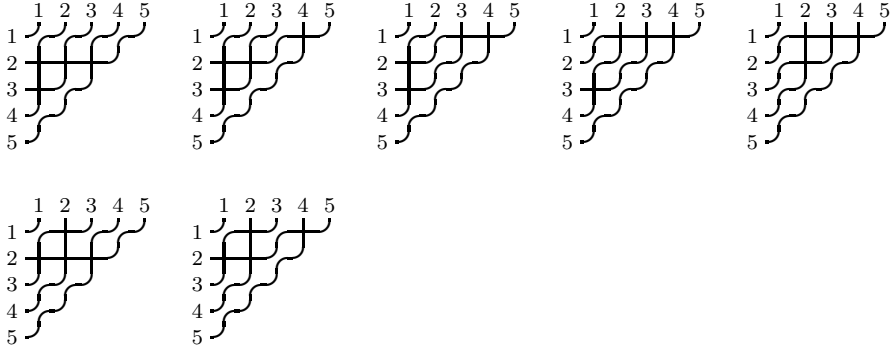

\begin{displaymath}
\begin{array}{ccccccc}
\pipes{ & 1 & 2 & 3 & 4 & 5 \\ 
1 & \turn & \turn & \turn & \turn & \tail \\
2 & \cross &  \cross & \cross& \tail \\
3 &\cross & \turn & \tail \\
4 & \turn & \tail \\
5 & \tail } &
\pipes{ & 1 & 2 & 3 & 4 & 5 \\ 
1 & \turn & \turn & \turn & \cross & \tail \\
2 & \cross &  \cross & \turn& \tail \\
3 &\cross & \turn & \tail \\
4 & \turn & \tail \\
5 & \tail } &
\pipes{ & 1 & 2 & 3 & 4 & 5 \\ 
1 & \turn & \turn & \cross & \cross & \tail \\
2 & \cross &  \turn & \turn& \tail \\
3 &\cross & \turn & \tail \\
4 & \turn & \tail \\
5 & \tail } &
\pipes{ & 1 & 2 & 3 & 4 & 5 \\ 
1 & \turn & \cross & \cross & \cross & \tail \\
2 & \turn &  \turn & \turn& \tail \\
3 &\cross & \turn & \tail \\
4 & \turn & \tail \\
5 & \tail } &
\pipes{ & 1 & 2 & 3 & 4 & 5 \\ 
1 & \turn & \cross & \cross & \cross & \tail \\
2 & \turn &  \cross & \turn& \tail \\
3 &\turn & \turn & \tail \\
4 & \turn & \tail \\
5 & \tail } \\ \\
\pipes{ & 1 & 2 & 3 & 4 & 5 \\ 
1 & \turn & \cross & \turn & \turn & \tail \\
2 & \cross &  \cross & \cross & \tail \\
3 &\turn & \turn & \tail \\
4 & \turn & \tail \\
5 & \tail } &
\pipes{ & 1 & 2 & 3 & 4 & 5 \\ 
1 & \turn & \cross & \turn & \cross & \tail \\
2 & \cross &  \cross & \turn& \tail \\
3 &\turn & \turn & \tail \\
4 & \turn & \tail \\
5 & \tail } 
\end{array}
\end{displaymath}
\caption{The 7 reduced pipe dreams associated to the permutation $15324$.}\label{fig:pipes}
\end{figure}

A reduced pipe dream $P$ is {\bf quasiYamanouchi} if the following is true for the leftmost $\cross$ in every row: Either
\begin{enumerate}
 \item it is in the leftmost column, or
\item it is weakly left of some $\cross$ in the row below it.
\end{enumerate}
For a permutation $\pi$, write $\QPD(\pi)$ for the set of quasiYamanouchi reduced pipe dreams for $\pi$. We obtain then the following formula for the fundamental slide polynomial expansion of a Schubert polynomial.

\begin{theorem}[\cite{Assaf.Searles}]\label{thm:Schub2slide}
The Schubert polynomials expand positively in the fundamental slide polynomials:
\[ \schub_\pi = \sum_{P \in \QPD(\pi)} \slide_{\wt(P)}.\]
\end{theorem}

Before continuing to our third combinatorial formula for Schubert polynomials, let us digress to consider another basis of $\poly_n$, closely related to the fundamental slide polynomials. Recall from Theorem~\ref{thm:slide_lift_fund} that the fundamental slide polynomials are a lift of a fundamental quasisymmetric polynomials. One might ask for an analogous lift to $\poly_n$ of the monomial quasisymmetric polynomials. These are provided by the {\bf monomial slide polynomials} of \cite{Assaf.Searles}, which we now discuss. 

Looking back at the combinatorial formulas for fundamental and monomial quasisymmetric polynomials, observe that they are identical, except that the formula for $\fundamental_\alpha$ looks at weak compositions $b$ with $b^+ \vDash \alpha$ while the formula for $M_\alpha$ looks at weak compositions $b$ with the more restrictive property $b^+  = \alpha$. It is easy then to guess the following modification of Theorem~\ref{thm:slide2X} that will yield the desired definition of monomial slide polynomials.

\begin{definition}\label{def:mslide2X}
For any weak composition $a$, the {\bf monomial slide polynomial} $\mslide_a$ is defined by
\[
\mslide_a = \sum_{\substack{ b \geq a \\ b^+ = a^+}} \monomial_b.
\]
\end{definition}

By triangularity, it is straightforward that monomial slide polynomials form a basis of $\poly_n$. Moreover, we have the following analogue of Theorem~\ref{thm:slide_lift_fund}:

\begin{theorem}[\cite{Assaf.Searles}]\label{thm:mslide_lift_mon}
We have $\mslide_a \in \qsym_n$ if and only if $a$ is of the form $0^k \alpha$, where $\alpha$ is a strong composition of length $n-k$.

Moreover, we have $\mslide_{0^k\alpha} = M_\alpha$. Thus, the monomial slide basis of $\poly_n$ is a lift of the monomial quasisymmetric polynomial basis of $\qsym_n$.
\end{theorem}

Just as the combinatorial multiplication rule of Theorem~\ref{thm:mult_F} for fundamental quasisymmetric polynomials lifts to that of Theorem~\ref{thm:mult_slide} for fundamental slide polynomials, the combinatorial multiplication rule of Theorem~\ref{thm:mult_M} for monomial quasisymmetric polynomials lifts to a rule for monomial slide polynomials.

First, we need to extend the overlapping shuffle product of Section~\ref{sec:qsym} from strong compositions to general weak compositions. Given weak compositions $a$ and $b$, treat them as words of of some common finite length $n$ by truncating at some position past all their nonzero entries. By $a+b$ we mean the weak composition that is the coordinatewise sum of $a$ and $b$. Consider the set $S(a,b)$ of all pairs $(a',b')$ of weak compositions of equal length $k \leq n$ such that 
\begin{itemize}
\item $(a')^+ = a^+$ and $(b')^+ = b^+$;
\item $a' \geq a$ and $b' \geq b$; and
\item for all $1 \leq i \leq k$, we have $a'_i + b'_i > 0$.
\end{itemize}

Fix $(a',b') \in S(a,b)$. Let $c$ be a weak composition of length $r$ with zeros in positions $s_1, \ldots , s_m$ such that $c^+ = a'+b'$. Define $c_a$ to be the weak composition of length $r$ having zeros in the same positions $s_1, \ldots , s_m$ and the remaining positions of $c_a$ are the entries of $a'$, in order from left to right. Define $c_b$ similarly, using the entries of $b'$. Then we have
\begin{itemize}
\item $c = c_a + c_b$, and
\item $(c_a)^+ = (a')^+$ and $(c_b)^+ = (b')^+$.
\end{itemize}
For each such $(a',b') \in S(a,b)$, let $\Bump(a',b')$ denote the unique dominance-least weak composition satisfying
\begin{itemize}
\item $\Bump(a',b')^+ = a' + b'$, and
\item $\Bump(a',b')_a \geq a$ and $\Bump(a',b')_b \geq b$.
\end{itemize}
The {\bf overlapping slide product} of $a$ and $b$ is then the formal sum $a \shuffle_o b$ of the $\Bump(a',b')$ for all $(a',b') \in S(a,b)$.

\begin{example}\label{ex:mult_mslide}
Let $a = (0,1,0,2)$ and $b = (1,0,0,1)$, as in Example~\ref{ex:mult_slide}. Then $S(a,b)$ consists of the seven pairs
\begin{align*}
&\big( (0,1,0,2), (1,0,1,0) \big), \big( (0,1,2,0), (1,0,0,1) \big),  \big( (0,1,2), (1,0,1) \big), \big( (0,1,2), (1,1,0) \big)
\\
&\big( (1,0,2), (1,1,0) \big), \big( (1,2,0), (1,0,1) \big),  \big( (1,2), (1,1) \big)
\end{align*}
The seven corresponding weak compositions $\Bump(a',b')$ are
\begin{align*}
& (1,1,1,2), (1,1,2,1), (1,1,0,3), (1,2,0,2), (2,0,1,2), (2,0,2,1), (2,0,0,3)
\end{align*}
\end{example}

\begin{theorem}[\cite{Assaf.Searles}]\label{thm:mult_mslide}
For weak compositions $a$ and $b$, we have 
\[
\mslide_a \cdot \mslide_b = \sum_c C_{a, b}^c \mslide_c,
\]
where $C_{a, b}^c$ is the multiplicity of $c$ in the overlapping slide product $a \shuffle_o b$.
\end{theorem} 

The fundamental slide polynomials expand positively in the monomial slide basis. Say that $b\unrhd a$ if $b\ge a$, and $c\ge b$ whenever $c\ge a$ and $c^+=b^+$. 

\begin{theorem}\cite{Assaf.Searles}
The fundamental slide polynomials expand positively in the monomial slide polynomials:
\[\slide_a = \sum_{\substack{b\unrhd a \\ b^+ \vDash a^+}} \mslide_b.\]
\end{theorem}

A third combinatorial formula for Schubert polynomials comes from a model introduced (conjecturally) by Axel Kohnert \cite{Kohnert}. Let $D$ be a box diagram, i.e., any subset of the boxes in an $n\times n$ grid. A {\bf Kohnert move} on $D$ selects the rightmost box in some row and moves it to the first available empty space above it in the same column (if such an empty space exists). Let $\KD(D)$ denote the set of all box diagrams that can be obtained from $D$ by some sequence (possibly empty) of Kohnert moves. Define the weight $\wt(D)$ of a box diagram to be the weak composition where $\wt(D)_i$ records the number of boxes in the $i$th row of $D$ from the top.

\begin{example}
Let $D$ be the leftmost diagram in the top row of Figure~\ref{fig:KD} (which happens to be the Rothe diagram $RD(15324)$). The set of diagrams in Figure~\ref{fig:KD} is exactly $\KD(D)$. 
\end{example}
\begin{figure}[h]
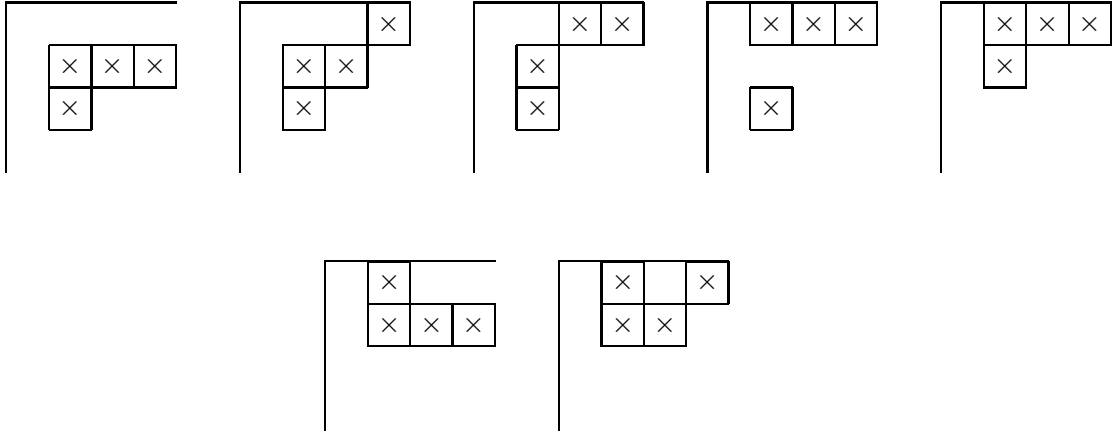

\begin{displaymath}
\vline \tableau{ \hline \\  & \times & \times & \times \\ & \times \\ \\ }  \hspace{1.5\cellsize} 
\vline \tableau{ \hline  & & & \times \\  & \times & \times  \\ & \times \\ \\ }  \hspace{1.5\cellsize} 
\vline \tableau{ \hline & & \times & \times \\  & \times   \\ & \times \\ \\ }  \hspace{1.5\cellsize} 
\vline \tableau{ \hline & \times & \times & \times \\ \\ & \times \\ \\ }  \hspace{1.5\cellsize}
\vline \tableau{ \hline & \times & \times & \times \\ & \times \\  \\ \\ }  
\end{displaymath}
\vspace{\baselineskip}
\begin{displaymath}
\vline \tableau{ \hline & \times & &  \\  & \times & \times & \times \\ \\ \\ }  \hspace{1.5\cellsize}
\vline \tableau{ \hline & \times & & \times  \\  & \times & \times  \\ \\ \\ }  \hspace{1.5\cellsize} 
\end{displaymath}
\caption{The $7$ Kohnert diagrams associated to $RD(15324)$.}\label{fig:KD}
\end{figure}

\begin{theorem}\cite{Win99,Win02,Assaf:comb}\label{thm:Kohnert_formula_for_Schubs}
The Schubert polynomial $\schub_\pi$ is the generating function of the Kohnert diagrams for the Rothe diagram of $\pi$, i.e.,
\[\schub_\pi = \sum_{D\in \KD(\RD(\pi))}\xx^{\wt(D)}.\]
\end{theorem}

\begin{example}
We have $\schub_{15324} = \xx^{031} + \xx^{121} +\xx^{211} +\xx^{310} +\xx^{310} +\xx^{130} +\xx^{220}$, 
where the monomials are determined by the Kohnert diagrams shown in Figure~\ref{fig:KD}. Note this is consistent with the computation in Example~\ref{ex:pipes}.
\end{example}

Kohnert diagrams also yield another natural basis of $\asym_n$, which we now consider. Given a weak composition $a$, we can associate a box diagram $D$ by first obtaining the permutation $w$ corresponding to $a$ and then taking $D = \RD(w)$. This construction, combined with the Kohnert moves, leads us to the characterization of Schubert polynomials from Theorem~\ref{thm:Kohnert_formula_for_Schubs}. However, there is also a much easier way to associate a box diagram $D(a)$ to a weak composition $a$--namely, just take $D(a)$ to be the Young diagram of $a$, as described in Section~\ref{sec:qsym}. This leads us to the following definition (really a theorem of Kohnert \cite{Kohnert}): For a weak composition $a$, the {\bf key polynomial} $\key_a$ is the generating function
\[\key_a \coloneqq \sum_{D\in KD(D(a))}\xx^{\wt(D)}.\]

\begin{example}\label{ex:KDkey}
Let $a=021$. Then $\key_a =  \xx^{021} + \xx^{111} + \xx^{201} + \xx^{210} + \xx^{120}$, as computed by the Kohnert diagrams in Figure~\ref{fig:KDkey}.
\end{example}
\begin{figure}[h]
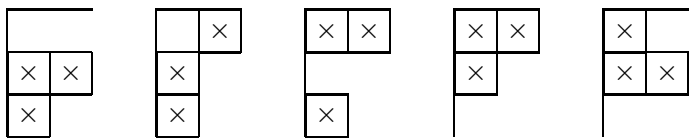

\begin{displaymath}
\vline \tableau{ \hline \\ \times & \times  \\ \times  \\ }  \hspace{1.5\cellsize} 
\vline \tableau{ \hline & \times \\ \times   \\ \times  \\ }  \hspace{1.5\cellsize} 
\vline \tableau{ \hline \times & \times \\   \\ \times  \\ }  \hspace{1.5\cellsize} 
\vline \tableau{ \hline \times & \times   \\ \times \\ \\ }  \hspace{1.5\cellsize} 
\vline \tableau{ \hline \times \\ \times & \times \\ \\ }
\end{displaymath}
\caption{The 5 Kohnert diagrams associated to the weak composition $021$.}\label{fig:KDkey}
\end{figure}

The definition of key polynomials that we have given here is not the original one. These polynomials were first introduced in \cite{Demazure} where they were realized as characters of (type A) Demazure modules; for this reason, they are often referred to as \emph{Demazure characters}. Later they were studied from a more combinatorial perspective by Lascoux and Sch\"utzenberger \cite{Lascoux.Schutzenberger:key}, who coined the term `key polynomial'. Key polynomials also arise as a specialization \cite{Sanderson, Ion} of the nonsymmetric Macdonald polynomials introduced in \cite{Opdam, Macdonald96, Cherednik}.

An alternative description of key polynomials is via a modification of the $\partial_i$ operators that define Schubert polynomials. For each positive integer, define an operator $\pi_i$ on $\poly_n$ by
\[
\pi_i(f) \coloneqq \partial_i(x_i f).
\]
Let $w$ be a permutation and let $s_{i_1}\cdots s_{i_r}$ a any reduced word for $w$. Then define $\pi_w = \pi_{i_1}\cdots \pi_{i_r}$. This is independent of the choice of reduced word since these operators satisfy $\pi_i^2=\pi_i$ and the usual commutation and braid relations for the symmetric group. (That is to say, the action of the $\pi_i$ operators on $\poly_n$ is a representation of the type A $0$-Hecke algebra.) Let $s_i$ act on a weak composition $a$ by exchanging the $i$th and $(i+1)$st entries of $a$. Given a weak composition $a$, let $w(a)$ denote the permutation of minimal Coxeter length such that $w(a)\cdot a = \sort{a}$. Finally, the key polynomial $\key_a$ is given by
\[\key_a = \pi_{w(a)} \xx^{\sort{a}}.\]

\begin{example}
\begin{align*}
\key_{021} & = \pi_1\pi_2(x_1^2x_2) \\
		& = \pi_1(x_1^2x_2+x_1^2x_3) \\
		& = x_1^2x_2+x_1x_2^2 + x_1^2x_3 + x_1x_2x_3 + x_2^2x_3.
\end{align*}
Compare this calculation with that of Example~\ref{ex:KDkey}.
\end{example}

As is the case for Schubert polynomials, there are several additional combinatorial formulas for key polynomials! An excellent overview can be found in \cite{Reiner.Shimozono}.
Another such combinatorial formula for the key polynomials is given in \cite{Haglund.Luoto.Mason.vanWilligenburg:refinement} in terms of \emph{semi-skyline fillings}. Let $a$ be a weak composition of length $n$.  We recall the definition of a triple of entries from the previous section; this extends verbatim from diagrams of compositions to diagrams of weak compositions.

A \emph{semi-skyline filling} of $D(a)$ is a filling of the boxes of $D(a)$ with positive integers, one per box, such that
\begin{enumerate}
\item[(S.1)] entries do not repeat in a column
\item[(S.2)] entries weakly decrease from left to right along rows
\item[(S.3)] every triple of entries is inversion.
\end{enumerate}
The weight $\wt(T)$ of a semi-skyline filling is weak composition whose $i$th entry records the number of occurrences of the entry $i$ in $T$.

Let $\overline{D}(a)$ denote the diagram of $D(a)$ augmented with an additional $0$th column called a \emph{basement}, and let $\rev(a)$ denote the weak composition obtained by reading the entries of $a$ in reverse. Define a \emph{key semi-skyline filling} for $a$ to be a filling of $\overline{D}(\rev(a))$, where the $i$th basement entry is $n+1-i$, satisfying (S.1), (S.2) and (S.3) above (including on basement entries). Let $\KSSF(a)$ denote the set of key semi-skyline fillings for $a$.

\begin{example}\label{ex:basement}
Let $a = (0,2,1)$.  The key semi-skyline fillings associated to $a$ are shown in Figure~\ref{fig:basement}. The basement boxes are shaded in grey.
\end{example}
\begin{figure}[ht]
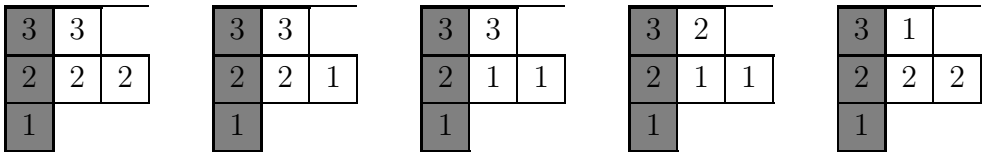

\begin{displaymath}
 \begin{ytableau} \hline
*(gray) 3 & 3 \\
*(gray) 2 & 2 & 2 \\
*(gray) 1 
\end{ytableau}  \hspace{1.5\cellsize} 
 \begin{ytableau} \hline
*(gray) 3 & 3 \\
*(gray) 2 & 2 & 1 \\
*(gray) 1 
\end{ytableau}  \hspace{1.5\cellsize} 
 \begin{ytableau} \hline
*(gray) 3 & 3 \\
*(gray) 2 & 1 & 1 \\
*(gray) 1 
\end{ytableau}  \hspace{1.5\cellsize} 
 \begin{ytableau} \hline
*(gray) 3 & 2 \\
*(gray) 2 & 1 & 1 \\
*(gray) 1 
\end{ytableau}  \hspace{1.5\cellsize} 
 \begin{ytableau} \hline
*(gray) 3 & 1 \\
*(gray) 2 & 2 & 2 \\
*(gray) 1 
\end{ytableau} 
\end{displaymath}
\caption{The five key semi-skyline fillings associated to $021$.}\label{fig:basement}
\end{figure}

\begin{theorem}\cite{Haglund.Luoto.Mason.vanWilligenburg:refinement}
The key polynomial $\key_a$ is given by
\[\key_a = \sum_{T\in \KSSF(a)}\xx^{\wt(T)}.\]
\end{theorem}
\begin{example}
From Example~\ref{ex:basement}, we again compute $\key_{021} =  \xx^{021} + \xx^{111} + \xx^{201} + \xx^{210} + \xx^{120}$.
\end{example}

Yet another formula for key polynomials is given in terms of \emph{Kohnert tableaux} \cite{Assaf.Searles:2}. Kohnert tableaux associate a canonical path in Kohnert's algorithm from the diagram of a weak composition to a given Kohnert diagram. In fact, the Kohnert tableaux are equivalent to the key semi-skyline fillings turned upside down, but are described by quite different local rules, which arise from Kohnert's algorithm as opposed to considerations in Macdonald polynomial theory. We omit the details of this construction, for which see \cite{Assaf.Searles:2}.

One might naturally ask, analogously for Schubert polynomials, whether key polynomials restrict to an important basis of $\sym_n \subset \asym_n$. Exactly the same as for the Schubert polynomials, the key polynomials in $\sym_n$ are exactly the basis of Schur polynomials. Thus the key basis of $\asym_n$ is also a lift of the Schur basis of $\sym_n$. The key polynomial $\key_a$ is a Schur polynomial if and only if $a$ is weakly increasing; in this case, we have $\key_a = s_{\sort{a}}$.

Moreover, the \emph{stable limits} of the key polynomials are exactly the Schur functions. Given $m\in \mathbb{Z}_{\ge 0}$ and a weak composition $a$, recall that $0^m a$ denotes the weak composition obtained by prepending $m$ zeros to $a$, e.g., $0^2 (1,0,3) = (0,0,1,0,3)$. The stable limit of $\key_a$ is the formal power series $\lim_{m\to \infty} \key_{0^m a}$. The following result is implicit in work of Lascoux and Sch\"utzenberger \cite{Lascoux.Schutzenberger:key}; an explicit proof is given in \cite{Assaf.Searles:2} with further details.

\begin{theorem}[{\cite{Lascoux.Schutzenberger:key,Assaf.Searles:2}}]
Let $a$ be a weak composition. Then the stable limit of the key polynomial $\key_a$ is the Schur function associated to the partition obtained by rearranging the entries of $a$ into decreasing order. That is,
\[\lim_{m\to \infty} \key_{0^m a} = s_{\sort{a}}(X).\]
\end{theorem}

Unlike the Schubert basis, however, the key basis does not have positive structure constants.

Schubert polynomials expand in the key basis with positive coefficients, although we omit the details of this decomposition. Let $T$ be a semistandard Young tableau. Define the \emph{column reading word}  ${\rm colword}(T)$ of $T$ to be the word obtained by writing the entries of each column of $T$ from bottom to top, starting with the leftmost column and proceeding rightwards. Then, as given in \cite[Theorem~4]{Reiner.Shimozono}, we have 
\[\schub_\pi = \sum_{{\rm colword}(T) \in \reduced(\pi^{-1})} \key_{\wt(K^0_-(T))}\]
where the sum is over all semistandard Young tableaux $T$ whose column reading word is a reduced word for $\pi^{-1}$, and $K^0_-(T)$ is the \emph{left nil key} of $T$, as defined in \cite{Lascoux.Schutzenberger:key, Reiner.Shimozono}. For another approach to this decomposition, see \cite{Assaf:EG}.

The Demazure atoms $\atom_a$ form another basis of $\asym_n$, introduced and studied in \cite{Lascoux.Schutzenberger:key}, where they are referred to as \emph{standard bases}. Demazure atoms are characters of quotients of Demazure modules, and, like key polynomials, also arise as specializations of nonsymmetric Macdonald polynomials \cite{Mason:RSK}. Just as we defined fundamental slide polynomials as the pieces of Schubert polynomials given by the summands of Theorem~\ref{thm:schub2X}, we can define the Demazure atoms as the pieces of quasiSchur polynomials given by the summands of the definition in Equation~\eqref{eq:quasiSchur}: Given a weak composition $a$, the {\bf Demazure atom} $\atom_a$ is given by
\[\atom_a \coloneqq \sum_{T\in \ASSF(a)} \xx^{\wt(T)}.\]

\begin{example}
We have 
\[\atom_{(1,0,3)} = \xx^{103} + \xx^{112} +\xx^{202} +\xx^{121} +\xx^{211},\]
where the monomials are determined by the five semistandard composition tableaux of shape $(1,0,3)$ from Figure~\ref{fig:ASSF}.
\end{example}

The Demazure atom basis does not have positive structure coefficients. However, it does exhibit a variety of surprising positivity properties.
First, notice that the definition that we have given immediately implies that the quasiSchur polynomial $\qschur_\alpha(x_1, \ldots , x_n)$ expands positively in Demazure atoms.

\begin{proposition}[\cite{Haglund.Luoto.Mason.vanWilligenburg:quasiSchur}]
The quasiSchur polynomials expand positively in the Demazure atoms:
\[\qschur_\alpha(x_1, \ldots , x_n) = \sum_{a^+ = \alpha} \atom_a,\]
where the sum is over weak compositions $a$ of length $n$.
\end{proposition}

It is also the case that key polynomials expand positively in Demazure atoms. Let $S_n$ act on weak compositions of length $n$ via $v\cdot (a_1, \ldots , a_n) = (a_{v^{-1}(1)},\ldots , a_{v^{-1}(n)})$. Given a weak composition $a$, let $v(a)$ denote the permutation of minimal Coxeter length such that $v(a) \cdot a = \sort{a}$.

\begin{theorem}[\cite{Lascoux.Schutzenberger:key}]\label{thm:keytoatom}
The key polynomials expand positively in the Demazure atoms:
\[\key_a = \sum_{\substack{v(b)\le v(a) \\ \sort{b}=\sort{a}}}\atom_b,\]
where $\le$ denotes the (strong) Bruhat order on permutations.
\end{theorem}

\begin{example}
Let $a=(1,0,3)$. Then $v(a) = 231$ and
\[\key_{(1,0,3)} = \atom_{(1,0,3)} + \atom_{(1,3,0)} + \atom_{(3,0,1)} + \atom_{(3,1,0)}.\] 
\end{example}

Finally, we mention the following remarkable conjecture of V.~Reiner and M.~Shimozono; for more details on this conjecture, see the work of A.~Pun \cite{Pun}. For a generalization, see \cite{Monical.Pechenik.Searles}. Observe that the conjecture would follow trivially from Theorem~\ref{thm:keytoatom} if either the key polynomial or the Demazure atom basis had positive structure coefficients; however, neither does, so the conjecture is quite mysterious.

\begin{conjecture}[Reiner--Shimozono]
The product $\key_a \cdot \key_b$ expands positively in Demazure atoms.
\end{conjecture}

At this point, we have considered lifts to $\asym_n$ of the Schur polynomials (two distinct lifts even), the fundamental quasisymmetric polynomials, and the monomial quasisymmetric polynomials. It is natural then to hope for an appropriate lift of the remaining basis of $\qsym_n$ that we considered in Section~\ref{sec:qsym}, namely the quasiSchur polynomials. The next basis we consider is exactly this desired lift, the quasikey polynomials of \cite{Assaf.Searles:2}.
The quasikey polynomials are a lifting of the quasiSchur basis of $\qsym_n$ to $\asym_n$, and simultaneously a common coarsening of the fundamental slide polynomial and Demazure atom bases. 

Let $a$ be a weak composition of length $n$. The {\bf quasikey polynomial} associated to $a$ is given by
\[\qkey_a \coloneqq \sum_{\substack{b^+=a^+ \\ b\ge a}} \sum_{T\in \ASSF(a)}\xx^{\wt(T)}\]
where the first sum is over all weak compositions $b$ of length $n$ satisfying $b\ge a$ in dominance order and whose positive part is $a^+$. The form of this definition is due to \cite{Searles}. The quasikey polynomials were originally defined in \cite{Assaf.Searles:2} as a weighted sum of \emph{quasi-Kohnert tableaux}; we omit this alternate formulation.

\begin{example}
Let $a=(1,0,3)$. Then 
\[\qkey_{(1,0,3)} = \xx^{130}+\xx^{220} + \xx^{103} + \xx^{112} + \xx^{202} + \xx^{121} + \xx^{211},\]
where the monomials are determined by the first seven semistandard composition tableaux shown in Figure~\ref{fig:ASSF}.
\end{example}

The definition we have given here immediately implies that each quasikey polynomial $\qkey_a$ expands positively in Demazure atoms. Specifically, we have the following.

\begin{theorem}[{\cite[Theorem 3.4]{Searles}}]\label{thm:qkey2atom}
The quasikey polynomials expand positively in the Demazure atoms:
For a weak composition $a$ of length $n$, we have
\[\qkey_a = \sum_{\substack{b^+=a^+ \\ b\ge a}} \atom_b,\]
where the sum is over all weak compositions $b$ of length $n$ satisfying $b\ge a$ in dominance order and whose positive part is $a^+$.
\end{theorem}

\begin{example}
One easily calculates from Theorem~\ref{thm:qkey2atom} that
\[\qkey_{(1,0,3)} = \atom_{(1,0,3)} + \atom_{(1,3,0)}.\]
Compare this calculation to the tableaux of Figure~\ref{fig:ASSF}.
\end{example}

Less clear from our definition is the following additional positivity property of quasikey polynomials.

\begin{theorem}[\cite{Assaf.Searles}]
The quasikey polynomials expand positively in the fundamental slide polynomials:
For a weak composition $a$ of length $n$, we have
\[\qkey_a = \sum_{T} \slide_{\wt(T)},\]
where the sum is over all 
quasiYamanouchi tableaux $T$ whose support contains the support of $a$, and such that $T \in \ASSF(b)$ for some weak composition $b$ of length $n$ with $b^+ = a^+$ and $b\ge a$.
\end{theorem}

\begin{example}
\[\qkey_{(1,0,3)} = \slide_{(1,0,3)} + \slide_{(2,0,2)}\]
where the two fundamental slides correspond to the $3$rd and $5$th composition tableaux in Figure~\ref{fig:ASSF}.
\end{example}

Although we claimed that the quasikey polynomials were to be a lift from $\qsym_n$ to $\asym_n$ of the quasiSchur polynomials, we have not yet explained this fact. The sense of this lift is given in the following proposition.

\begin{proposition}\cite[Theorem 4.16]{Assaf.Searles:2}\label{prop:qkeyisqschur}
We have $\qkey_a \in \qsym_n$ if and only if $a$ is of the form $0^k \alpha$, where $\alpha$ is a strong composition of length $n-k$. Moreover, we have
\[\qkey_{0^k \alpha} = \qschur_{\alpha}(x_1,\ldots , x_n).\] Thus, the quasikey polynomial basis of $\asym_n$ is a lift of the quasiSchur polynomial basis of $\qsym_n$.
\end{proposition}
Moreover, the quasikey polynomials stabilize to the quasiSchur functions:

\begin{theorem}\cite[Theorem 4.17]{Assaf.Searles:2}
Let $a$ be a weak composition. Then
\[\lim_{m\to \infty} \qkey_{0^m a} = \qschur_{a^+}.\]
\end{theorem}

As Schur polynomials expand positively in the quasiSchur polynomial basis, one might hope for the same positivity to hold for their respective lifts, the key polynomials (or Schubert polynomials) and quasikey polynomials. Indeed, these expansions are positive, with positive combinatorial formulas mirroring the expansion of Theorem~\ref{thm:s2S}. To provide a formula for this expansion, we need the concept of a left swap on weak compositions. A {\bf left swap} on a weak composition $a$ exchanges two entries $a_i$ and $a_j$ such that $a_i < a_j$ and $i<j$. In essence, left swaps move larger entries leftwards. Given a weak composition $a$ of length $n$, define the set $\lswap(a)$ to be all weak compositions $b$ of length $n$ that can be obtained from $a$ by a sequence of left swaps.

\begin{example}
\[\lswap(1,0,3) = \{(1,0,3), (1,3,0), (3,0,1), (3,1,0)\}.\]
\end{example}

In fact, the elements of $\lswap(a)$ are exactly those weak compositions $b$ such that $\sort{b}=\sort{a}$ and $w(b)\le w(a)$ in Bruhat order. Hence, in this new language, the formula in Theorem~\ref{thm:keytoatom} for expanding key polynomials in Demazure atoms may be re-expressed as follows.
\begin{proposition}[{\cite[Lemma 3.1]{Searles}}]
\[\key_a = \sum_{b\in \lswap(a)}\atom_b\]
\end{proposition}

Define $\Qlswap(a)$ to be those $b\in \lswap(a)$ such that for all $c\in \lswap(a)$ with $c^+=b^+$, one has $c\ge b$ in dominance order.
\begin{example}
\[\Qlswap(1,0,3) = \{(1,0,3), (3,0,1\}.\]
\end{example}

\begin{theorem}[{\cite[Theorem 3.7]{Assaf.Searles:2}}]\label{thm:keytoqkey}
Let $a$ be a weak composition of length $n$. Then 
\[\key_a = \sum_{b\in \Qlswap(a)} \qkey_b.\]
\end{theorem}

A key polynomial $\key_a$ is a Schur polynomial if and only if the entries of $a$ are weakly increasing. In this case, by Proposition~\ref{prop:qkeyisqschur} and the definition of $\Qlswap$, the formula of Theorem~\ref{thm:keytoqkey} reduces to the formula of Theorem~\ref{thm:s2S} for the quasiSchur polynomial expansion of a Schur polynomial.

Quasikey polynomials do not have positive structure constants. This is immediate from the fact that the quasikey polynomial basis contains all the quasiSchur polynomials, which themselves do not have positive structure constants. However, there is a positive combinatorial formula (which we do not describe here) for the quasikey expansion of the product of a quasikey polynomial and a Schur polynomial \cite{Searles}, extending the analogous formula of \cite{Haglund.Luoto.Mason.vanWilligenburg:refinement} for quasiSchur polynomials.

The final basis of $\asym_n$ that we consider here is the basis of \emph{fundamental particles} introduced in \cite{Searles}. 
Note that the formula for the Demazure atom expansion of a quasikey polynomial given in Theorem~\ref{thm:qkey2atom} is identical to the formula for the monomial expansion of a monomial slide polynomial given in Definition~\ref{def:mslide2X}. The main motivating property of fundamental particles is that this same formula will give the fundamental particle expansion of a fundamental slide polynomial. 

Given a weak composition $a$, let $\LSSF(a)$ denote the set of those semistandard composition tableaux of shape $a$ satisfying the property that whenever $i<j$, every label in row $i$ is smaller than every label in row $j$. For example, $\LSSF(1,0,3)$ consists of the 3rd, 4th and 6th tableaux in Figure~\ref{fig:ASSF}.

Let $a$ be a weak composition of length $n$. The {\bf fundamental particle} associated to $a$ is given by
\[\particle_a \coloneqq \sum_{T\in \LSSF(a)}\xx^{\wt(T)}.\]

\begin{example}
Let $a=(1,0,3)$. Then 
\[\particle_{(1,0,3)} = \xx^{103} + \xx^{112} + \xx^{121}\]
\end{example}

The following is straightforward from the definitions.
\begin{theorem}[\cite{Searles}]\label{thm:slide2particle}
The fundamental slide polynomials expand positively in the fundamental particles:
For a weak composition $a$ of length $n$, we have
\[\slide_a = \sum_{\substack{b^+=a^+ \\ b\ge a}} \particle_b,\]
where the sum is over all weak compositions $b$ of length $n$ satisfying $b\ge a$ in dominance order and whose positive part is $a^+$.
\end{theorem}

The fundamental particles were constructed to be a refinement of the fundamental slide polynomials (with a particular positive expansion formula), as in Theorem~\ref{thm:slide2particle}; remarkably, the fundamental particles are also a refinement of the Demazure atoms. Given a weak composition $a$, define the set $\HSSF(a)$ of {\bf particle-highest} semistandard composition tableaux of shape $a$ to be the set of those $T\in \ASSF(a)$ such that for each integer $i$ appearing in $T$, either
\begin{itemize}
\item an $i$ appears in the first column, or
\item there is an $i^+$ weakly right of an $i$, where $i^+$ is the smallest integer greater than $i$ appearing in $T$.
\end{itemize}
Notice the particle-highest condition is a weakening of the quasiYamanouchi condition: every quasiYamanouchi tableau is necessarily particle-highest.

Establishing the last of the arrows shown in Figure~\ref{fig:poly}, we have the following additional positivity.
\begin{theorem}[\cite{Searles}]
The Demazure atoms expand positively in the fundamental particles:
\[\atom_a = \sum_{T\in \HSSF(a)} \particle_{\wt(T)}.\]
\end{theorem}

Fundamental particles do not have positive structure constants, however, in analogy with quasikey polynomials and Demazure atoms, there is a positive combinatorial formula for the fundamental particle expansion of the product of a fundamental particle and a Schur polynomial given in \cite{Searles}.

\section{The mirror worlds: $K$-theoretic polynomials}

A trend in modern Schubert calculus is to look at $\mathrm{Flags}_n$, Grassmannians, and other generalized flag varieties, not through the lens of ordinary cohomology as in Sections~\ref{sec:sym}--\ref{sec:poly}, but through the sharper yet more mysterious lenses of other complex oriented cohomology theories \cite{Ganter.Ram,Calmes.Zainoulline.Zhong,Lenart.Zainoulline:elliptic, Lenart.Zainoulline:root}.

Particularly well studied over the past 20 years are combinatorial aspects of the $K$-theory rings of these spaces. Early work here includes \cite{Lascoux.Schutzenberger,Lascoux.Schutzenberger:symmetry,Fomin.Kirillov,Fulton.Lascoux}; however, the area only became very active after the influential work of \cite{Lenart, Buch}.
Following \cite{Bressler.Evens, Fomin.Kirillov, Hudson}, it turns out that one can slightly generalize this setting to \emph{connective $K$-theory} with almost no extra combinatorial complexity (indeed, in some ways the more general combinatorics seems easier). In general, each complex oriented cohomology theory is determined by its formal group law, which describes how to write the Chern class of a tensor product of two line bundles in terms of the two original Chern classes. In the case of connective $K$-theory, the formal group law is
\begin{equation}\label{eq:fgl}
c_1(L \otimes M) = c_1(L) + c_1(M) + \beta c_1(L) c_1(M),
\end{equation}
where $\beta$ is a formal parameter and $L,M$ are complex line bundles on the space in question. In this notation, the ordinary cohomology ring is recovered by setting $\beta = 0$ and the ordinary $K$-theory ring is recovered (up to convention choices) by setting $\beta = -1$.

Just as the Schubert classes in the ordinary cohomology of $\mathrm{Flags}_n$ are represented by the Schubert polynomials (as described in Section~\ref{sec:poly}), we would like to have such polynomial representatives for the corresponding connective $K$-theory classes.
These are provided by the \emph{$\beta$-Grothendieck polynomials} $\{ \groth_a \}$ of S.~Fomin and A.~Kirillov \cite{Fomin.Kirillov}, as identified in \cite{Hudson}. These polynomials form a basis of $\poly_n[\beta]$, where $\beta$ is the formal parameter from Equation~(\ref{eq:fgl}). This basis is homogeneous if the parameter $\beta$ is understood to live in degree $-1$. Specializing at $\beta =0$, one recovers the Schubert basis $\{ \schub_a \}$ of $\poly_n$.  The usual Grothendieck polynomials of A.~Lascoux and M.-P.~Sch\"{u}tzenberger \cite{Lascoux.Schutzenberger} are realized at $\beta = -1$. (To help the reader track relations among bases, we deviate from established practice by denoting connective $K$-analogues by applying an `overbar' to their cohomological specializations.) Like the Schubert basis of $\poly_n$, the $\beta$-Grothendieck polynomial basis has positive structure coefficients; this is, of course, currently only known by geometric arguments \cite{Brion}, as there is no combinatorial proof of this fact even in the case $\beta=0$.

Intersecting $\{ \groth_a \}$ with $\sym_n[\beta]$ yields the basis $\{ \sgroth_\lambda \}$ of \emph{symmetric Grothendieck polynomials}. These represent connective $K$-theory Schubert classes on Grassmannians. In this setting, like the Schur polynomial setting, a number of Littlewood-Richardson rules for $\{ \sgroth_\lambda \}$ are now known (e.g., \cite{Vakil,Thomas.Yong:K,Pechenik.Yong:genomic}), following the first found by A.~Buch \cite{Buch}. 

The remaining families of polynomials discussed in Sections~\ref{sec:sym}--\ref{sec:poly} are not currently understood well in term of cohomology. Remarkably, however, from a combinatorial perspective they all appear to have natural `connective $K$-analogues'. That is, for each basis, there is a combinatorially-natural $\beta$-deformation that is homogeneous (with the understanding that $\beta$ has degree $-1$), forms a basis of $\poly_n[\beta]$, and (at least conjecturally) shares the positivity properties of the original basis. These deformed bases are presented in Table~\ref{tab:mytab}. 

It is a mystery as to whether these various apparently $K$-theoretic families of polynomials in fact have an geometric interpretation in terms of $K$-theory. If they did, then presumably the $\beta=0$ specialization considered in the previous sections would similarly have a cohomological interpretation. Such an interpretation would be rather surprising, as currently these specializations are only understood through combinatorics and (in some cases) representation theory. Alternatively, perhaps there is a general representation-theoretic construction that yields all of these various $\beta$-deformations. No such construction is currently known, but for some ideas along these lines see \cite{Galashin, Monical.Pechenik.Scrimshaw, Pechenik:genomic}.

\begin{table}[htb]
\scalebox{0.7}{
\begin{tabular}{|l|l|l|}
\hline
$\scolor{\sym_n}$ & Schur polynomial $s_\lambda$ & symmetric Grothendieck polynomial $\sgroth_\lambda$ \cite{Buch, Monical.Pechenik.Scrimshaw} \\
\hline
 $\qcolor{\qsym_n}$	& monomial quasisymmetric polynomial $M_\alpha$ & multimonomial polynomial $\overline{M}_\alpha$ \cite{Lam.Pylyavskyy} \\
 & fundamental quasisymmetric polynomial $F_\alpha$ & multifundamental polynomial $\overline{F}_\alpha$ \cite{Lam.Pylyavskyy, Patrias, Pechenik.Searles} \\
 & quasiSchur polynomial $S_\alpha$ & quasiGrothendieck polynomial $\overline{S}_\alpha$ \cite{Monical, Monical.Pechenik.Searles} \\
 \hline
 $\pcolor{\poly_n}$ & Schubert polynomial $\schub_a$ & Grothendieck polynomial $\groth_a$ \cite{Lascoux.Schutzenberger, Fomin.Kirillov, Knutson.Miller} \\
 & Demazure character/key polynomial $\key_a$ & Lascoux polynomial $\lascoux_a$ \cite{Ross.Yong, Kirillov:notes, Monical, Monical.Pechenik.Searles} \\
 & quasikey polynomial $\qkey_a$ & quasiLascoux polynomial $\qlascoux_a$ \cite{Monical.Pechenik.Searles} \\
 & Demazure atom/standard basis $\atom_a$ & Lascoux atom $\lascouxatom_a$ \cite{Monical, Monical.Pechenik.Searles} \\
 & fundamental particle/pion $\particle_a$ & kaon $\kaon_a$ \cite{Monical.Pechenik.Searles} \\
 & fundamental slide polynomial $\slide_a$ & glide polynomial $\glide_a$ \cite{Pechenik.Searles, Monical.Pechenik.Searles} \\
\hline
\end{tabular}}
\caption{Bases from Sections~\ref{sec:sym}--\ref{sec:poly}, together with their corresponding connective $K$-analogues. For each $K$-theoretic family of polynomials, we have a given a few major references; however, these references are generally not exhaustive.}
\label{tab:mytab}
\end{table}

For details of the definitions of thebases from Table~\ref{tab:mytab} and their relations, see the references given there, especially \cite{Monical.Pechenik.Searles} which contains a partial survey. Here, we only briefly sketch hints of this theory. (However, the theory is in many ways exactly parallel to that given in Sections~\ref{sec:sym}--\ref{sec:poly}, so the astute reader can likely guess approximations to many of the structure theorems.)

The notion of semistandard \emph{set-valued} skyline fillings was introduced in \cite{Monical}, and employed there to provide an explicit combinatorial definition of the \emph{Lascoux polynomials} $\lascoux_a$ and the \emph{Lascoux atoms} $\lascouxatom_a$: $K$-theoretic analogues of key polynomials and Demazure atoms, respectively.  Analogous $K$-theoretic analogues of key polynomials have also been studied in \cite{Lascoux:transition,Ross.Yong,Kirillov:notes}. 

The basis $\glide_a$ of \emph{glide polynomials} was introduced in \cite{Pechenik.Searles}. Glide polynomials are simultaneously a $K$-theoretic analogue of the fundamental slide basis and a polynomial lift of the \emph{multi-fundamental quasisymmetric} basis \cite{Lam.Pylyavskyy} of quasisymmetric polynomials. Remarkably, the glide basis also has positive structure constants, which can be described in terms the \emph{glide product} \cite{Pechenik.Searles} of weak compositions. The glide product is a simultaneous generalization of the slide product of \cite{Assaf.Searles} and the \emph{multi-shuffle} product of \cite{Lam.Pylyavskyy} on strong compositions.

The \emph{quasiLascoux} basis $\qlascoux_a$ and \emph{kaon} basis $\kaon_a$ were introduced in \cite{Monical.Pechenik.Searles}. These are $K$-analogues of the quasi-key and fundamental particle bases. Remarkably, the positivity relations between the bases in Figure~\ref{fig:poly} have been proven to hold for their $K$-analogues mentioned above \cite{Pechenik.Searles, Monical.Pechenik.Searles}, with the exception of the expansion of Grothendieck polynomials in Lascoux polynomials, whose positivity remains conjectural.

As outlined in the previous section, the product of an element of any basis in Figure~\ref{fig:poly} with a Schur polynomial expands positively in that basis. It would be interesting to know if the analogous result is true in the $K$-theory world: that product of an element of a $K$-theoretic analogue and a symmetric Grothendieck polynomial expands positively in that $K$-theoretic basis. This is obviously true for the Grothendieck basis by geometry. It is also true for the glide basis, since the glide basis has positive structure constants and refines Grothendieck polynomials. We believe this question remains, however, open for the Lascoux, quasiLascoux, Lascoux atom and kaon bases.

\section*{Acknowledgements}
OP was partially supported by a Mathematical Sciences Postdoctoral Research Fellowship (\#1703696) from the National Science
Foundation.

We are grateful to the many people who taught us about polynomials and, over many years, have influenced the way we think about the three worlds. It seems hopeless to provide a full accounting of our influences, but in particular we wish to thank Sami Assaf, Alain Lascoux, Cara Monical, Stephanie van Willigenburg, and our common advisor Alexander Yong.

%
%

\bibliographystyle{amsalpha} 
\bibliography{polynomialSurvey}

\end{document}